\documentclass[11pt,a4paper,twoside, abstracton
]{scrartcl}
\usepackage{myscrartcl}
\usepackage[T1]{fontenc}

\usepackage{srcltx} %

\usepackage{enumerate} %
\newenvironment{packedenum}{
 \begin{enumerate}
 \setlength{\itemsep}{1pt} \setlength{\parskip}{0pt}
 \setlength{\parsep}{0pt} }{\end{enumerate}}
\newenvironment{packedenuma}{
 \begin{enumerate}[a.]
 \setlength{\itemsep}{1pt} \setlength{\parskip}{0pt}
 \setlength{\parsep}{0pt} }{\end{enumerate}}

\usepackage{pstricks,pst-node,pst-text,pst-3d, pst-plot}

\usepackage{amsmath, amsfonts}

\usepackage{amsthm}

\theoremstyle{plain}
\newcounter{thm}[section]

\newtheorem{theorem}[thm]{Theorem}
\newtheorem{corollary}[thm]{Corollary}
\newtheorem{lemma}[thm]{Lemma}

\newcounter{ass}[section]

\newtheorem{assumption}[ass]{Assumption}

\newcounter{defi}[section]

\newtheorem{definition}[defi]{Definition}

\newcounter{rem}[section]

\newtheorem{remark}[rem]{Remark}

\marginparwidth18mm \def\myproof[#1]{\par{\it Proof. #1} \ignorespaces}

\def\R{\mathbb R}
\newcommand{\fibcot}{({\mathbb R}^n)^* \times {\mathbb R}^n}
\newcommand{\cotsp}{\left( \R^n \right)^*}
\newcommand{\spa}{\R^n}
\newcommand{\cotbun}{T^*\R^n}
\newcommand{\tansp}{\R^n}
\newcommand{\e}{\varepsilon}
\newcommand\raggio{\rho}
\newcommand{\qo}{\text{ a.e.~}}
\newcommand\CS{V}

\newcommand\argmin{\operatorname{argmin}}

\newcommand\spd[2]{\dfrac{\partial#1}{\partial#2}}

\newcommand{\Span}{\operatorname{span}} \newcommand{\N}{\mathbb N}
 
\renewcommand{\a}{\alpha}

\newcommand{\se} {^{\prime \prime}}
\newcommand{\cinf}{C^\infty}
\newcommand{\ext}{\textrm{ext}}

\newcommand{\sgn}[1]{{\rm sgn}\left( #1 \right)}
\newcommand{\abs}[1]{\left\vert#1\right\vert}
\newcommand{\norm}[1]{\Vert#1\Vert}
\newcommand{\direz}[1]{\dfrac{\partial}{\partial {#1}}}

\newcommand{\dsum}{\displaystyle\sum}
 
\newcommand{\ud}{\operatorname{d}\!}
\newcommand{\uD}{\operatorname{D}\!}

\newcommand{\xJ}{{\^ x_1}}
\newcommand{\lf}{\widehat\ell_f}
\newcommand{\dl}{{\delta\ell}}
\newcommand{\dtau}{{\delta\tau}}
\newcommand{\dT}{{\delta T}}

\newcommand{\dx}{{\delta x}}

\newcommand{\scal}[2]{\langle {#1} \, , \, {#2} \rangle}
\newcommand{\dueforma}[2]{{\boldsymbol\sigma}\left({#1},{#2}\right)}
\newcommand{\bsi}{{\boldsymbol\sigma} }
\newcommand{\unoforma}{{\boldsymbol s}}

\newcommand{\gi}[2]{g_{{#1}, \,{#2}}}
\newcommand{\dgi}[2]{{\dot g}_{{#1}, \,{#2}}}
\newcommand{\gir}[2]{g^{r}_{{#1}, \,{#2}}}
\newcommand{\dgir}[2]{{\dot g^{r}}_{{#1}, \,{#2}}}
\newcommand{\dgin}[2]{{\dot g^{r_n}}_{{#1}, \,{#2}}}

\newcommand{\Sigmar}{\Sigma^r}

\newcommand{\cA}{{\mathcal A}}

\newcommand{\cF}{{\mathcal F}}

\newcommand{\cG}{{\mathcal G}}
\newcommand{\cH}{{\mathcal H}}

\newcommand{\cFS}{{\mathcal F}^{S}}

\newcommand{\cO}{{\mathcal O}}
\newcommand{\cS}{{\mathcal S}}

\newcommand{\cU}{{\mathcal U}}
\newcommand{\cV}{{\mathcal V}}
\newcommand{\cW}{{\mathcal W}}

\newcommand{\vchi}{\overrightarrow{\chi}}

\newcommand{\vFref}[1]{\overrightarrow{\^ F_{#1}}}
\newcommand{\vF}[1]{\overrightarrow{F_{#1}}}
\newcommand{\Fr}[1]{{F^r_{#1}}}
\newcommand{\vFr}[1]{\overrightarrow{F^r_{#1}}}
\newcommand{\vFS}{\overrightarrow{F^S}}

\newcommand{\vFSr}{\overrightarrow{F^{S, r}}}

\newcommand{\vG}{\overrightarrow{G}}
\newcommand{\vH}[1]{\overrightarrow{H_{#1}}}
\newcommand{\vHr}[1]{\overrightarrow{H^r_{#1}}}

\newcommand{\Maxdd}{M_1}
\newcommand{\mindd}{m_1}
\newcommand{\minbang}{\alpha_1}
\newcommand{\Maxddse}{M_2}
\newcommand{\minddse}{m_2}
\newcommand{\minbangse}{\alpha_2}

\newcommand\inv{^{-1}}
\renewcommand\*{\circ}
\newcommand\ol{\overline}

\newcommand{\Pbr}{\mathbf{(P_r)}}
\newcommand{\Pbz}{\mathbf{(P_0)}}
\renewcommand\^{\widehat }
\renewcommand\~{\widetilde}

\newcommand{\lri}{\widehat\ell_0}
\newcommand{\mri}{\widehat\mu_0}
\newcommand{\lrp}{\widehat\ell_1}
\newcommand{\mrp}{\widehat\mu_1}
\newcommand{\lrs}{\widehat\ell_2}
\newcommand{\mrs}{\widehat\mu_2}

\usepackage{preview}

\title{Structural stability for bang--singular--bang extremals in the
 minimum time problem \thanks{This work was partially supported by
 PRIN 200894484E\_002, Controllo Nonlineare: metodi geometrici e applicazioni} }

\author{Laura Poggiolini and Gianna Stefani\thanks{Dipartimento di
 Sistemi e Informatica -- Universit\`a degli Studi di Firenze,
 Italy ({\tt laura.poggiolini@unifi.it}, {\tt
 gianna.stefani@unifi.it}).} }

\date{}

\begin{document}

\maketitle
\begin{abstract}
  In this paper we study the structural stability of a
  bang-singular-bang extremal in the minimum time problem between
  fixed points.  The dynamics is single-input and control-affine.

  On the nominal problem ($r = 0$), we assume the coercivity of a
  suitable second variation along the singular arc and regularity both
  of the bang arcs and of the junction points, thus obtaining the
  strict strong local optimality for the given bang-singular-bang
  extremal trajectory. Moreover, as in the classically studied regular
  cases, we assume a suitable controllability property, which grants
  the uniqueness of the adjoint covector.

  Under these assumptions we prove that, for any sufficiently small
  $r$, there is a bang-singular-bang extremal trajectory which is a
  strict strong local optimiser for the $r$-problem. A uniqueness result
  in a neighbourhood of the graph of the nominal extremal pair is also
  obtained.

  The results are proven via the Hamiltonian approach to optimal
  control and by taking advantage of the implicit function theorem, so
  that a sensitivity analysis could also be carried out.
\end{abstract}
{\bfseries Keywords}: 
  Hamiltonian methods, second variation, structural stability.

\section{Introduction}
\label{sec:intro}
Since in practical optimisation problems the values of the data
usually are not known exactly and/or are subject to disturbances,
stability and sensitivity analysis constitute a crucial element of the
so-called post-optimisation analysis, which helps to evaluate the
practical usefulness of the obtained results.
 
Here we study the structural stability of a bang--singular--bang
extremal in the minimum time problem where the dynamics is
single-input and control-affine.  The paper is based on the
Hamiltonian approach which is used both in the optimality and in the
stability results.

We point out that, as in the classically studied regular cases (see
\cite{Mln93, Mln94, Mln01}), the assumptions on the nominal problem
are the ones which give optimality, see \cite{PS08, PS11}, together
with a controllability assumption which grants the uniqueness of the
adjoint covector.

The parameter-dependent minimum time problem $\Pbr$ we study is given
by
\begin{align}
  &\dot \xi^r(t) = f_0^r(\xi(t)) + u(t) f_1^r(\xi(t))&\label{eq:2r}\\
  & u(t) \in [-1, 1] \label{eq:4r}
\end{align}
and is constrained to
\begin{equation}\label{eq:3r}
  \xi^r(0) = a^r \, , \quad \xi^r(T) = b^r\,,
\end{equation}
where $a^r$ and $b^r$ are two given points. The parameter $r$ is in
$\R^m$, the state space is $\R^n$ (but the result can be easily
generalised to the case when the state space is a smooth finite
dimensional manifold) and all the data are assumed to be smooth, say
$\cinf$.

We study two different kinds of strong local optimality of a triplet
$(T^r, \xi^r, u^r)$ which is admissible for $\Pbr$ according to the
following definitions
\begin{definition}
  The trajectory $\xi^r \colon [0, T^r] \to \spa$ is a {\em (time,
    state)--local minimiser} of $\Pbr$ if there is a neighbourhood
  $\~\cU$ of its graph in $\R \times \spa$ and $\e >0$ such that
  $\xi^r$ is a minimiser among the admissible trajectories whose
  graphs are in $\~\cU$ and whose final time is greater than $T^r -
  \e$, independently of the values of the associated controls.
\end{definition}

We point out that this kind of optimality is local both with respect
to time and space. A {\em stronger version} of strong local optimality
is the so--called {\em state--local optimality} which is defined as
follows:
\begin{definition}
  The trajectory $\xi^r$ is a {\em state--local minimiser} of $\Pbr$
  if there is a neighbourhood $\cU$ of its range in $\spa$ such that
  $\xi^r$ is a minimiser among the admissible trajectories whose range
  is in $\cU$, independently of the values of the associated controls.
\end{definition}

For the nominal problem ($r=0$), we assume the coercivity of a
suitable second variation along the singular arc and regularity both
of the bang arcs and of the junction points, thus obtaining the strict
state local optimality for the given bang-singular-bang extremal, and
a suitable controllability assumption along the singular arc only, see
Section \ref{sec:assnominal}.

Under these assumptions we shall prove that, for any sufficiently
small $r \in \R^m$, there is a bang-singular-bang extremal trajectory
$\xi^r$ which is a strict strong local optimiser for problem $\Pbr$.
Moreover, if $\mu^r$ is the costate associated to $\xi^r$, then there
exists a neighbourhood $\cV$ of the graph of the nominal pair
$\^\lambda = \left( \^\mu, \^\xi \right)$ such that $\left( \mu^r ,
  \xi^r \right)$ is the only extremal pair of $\Pbr$ whose graph is in
$\cV$.

The results are proven via the Hamiltonian approach to optimal control
and by taking advantage of the implicit function theorem. Thus the
trajectory $\xi^r$ and its switching times depend smoothly on the
parameter $r$, so that a sensitivity analysis could also be carried
out.

For the regular cases we refer to \cite{Mln93, Mln94, Mln01} and the
references therein.  For control affine dynamics we mention
\cite{Fel04, FPS09, PSp08, PSp11} where bang-bang extremals for the
nominal problem are considered.  Bang-singular-bang extremals for the
Mayer problem are studied also in \cite{Fel12, Fel11} where the
author, under suitable assumptions, shows that if the perturbed
problem has an extremal which is some sense {\em near} the reference
one, then this extremal has the same bang-singular-bang structure.

We assume we are given a reference triplet $(\^T, \^\xi, \^u)$ which
is a normal bang--singular--bang Pontryagin extremal for the nominal
problem $\Pbz$ that is $\^u$ has the following structure
\begin{equation}
  \label{eq:refcontrol}
  \begin{alignedat}{2}
    &\^ u(t) \equiv u_1 \in \{ -1, 1 \} \qquad && \forall t \in [0,
    \^\tau_1) \, , \\
    & \^ u(t) \in(-1,1) \qquad && \forall t \in (\^\tau_1, \^\tau_2)
    \, , \\
    &\^ u(t) \equiv u_2 \in \{ -1, 1 \} \qquad && \forall t \in (
    \^\tau_2, \^T] \, .
  \end{alignedat}
\end{equation}
so that the {\em reference vector field} driving the nominal system is given by
\[
\^{f}_t =\begin{cases}
  h_1 := f_0 + u_1 f_1 &\mbox{if }\ t \in [0,\^\tau_1 )\\
  f_0 + \^ u(t) f_1 &\mbox{if }\ t \in (\^\tau_1 ,\^\tau_2) \\
  h_2 := f_0 + u_2 f_1 &\mbox{if }\ t \in (\^\tau_2, \^T ]
\end{cases}
\]
We shall refer to $\^\tau_1$, $\^\tau_2$ as to the {\em switching
  times} of the reference control $\^ u$.

The plan of the paper is as follows: we conclude this section by
giving the fundamental notation. In Section \ref{sec:assnominal} we
state the assumptions on the nominal problem; the regularity
assumptions are stated in Section \ref{ass:PMP} while the coercivity
and the controllability assumptions are stated in Sections
\ref{sec:ejse} and \ref{sec:consequences}. In Section \ref{sec:main}
we give the main results and an example. Finally in Section
\ref{sec:proof} we give all the proofs of the main results.
\subsection{Notation}\label{sub:notation}
In this paper we use some basic element of the theory of symplectic
manifolds for the cotangent bundle $\cotbun = \cotsp \times \spa$. For
a general introduction see \cite{Arn80}, for specific application to
Control Theory see e.g.~\cite{AS04}. Let us recall some basic facts
and let us introduce some specific notations.

We denote by $ \pi \colon \ell = (p, q) \in \cotbun \mapsto q \in
\spa$ the canonical projection.  If $V \subset \spa$ we denote as
$V^\perp \subset \cotsp$ its orthogonal space. The symbol $\unoforma$
denotes the canonical Liouville one--form on $\cotbun$: $\unoforma :=
\sum_{i=1}^n p^i \ud q_i$.  The associated canonical symplectic
two--form $\bsi = \ud\unoforma$ allows one to associate to any,
possibly time-dependent, smooth Hamiltonian $H_t \colon \cotbun
\rightarrow \R$, a Hamiltonian vector field $\vH{t}$, by
\[
\bsi(v,\vH{t}(\ell))=\scal{ \ud H_t(\ell)}{v} ,\quad \forall v\in
T_{\ell}\cotbun .
\]
In coordinates
\[
\vH{t}(\ell) \left( \left.  - \, \dfrac{\partial H_t}{\partial q}
  \right\vert_{\ell}, \left. \dfrac{ \partial H_t}{\partial p}
  \right\vert_{\ell} \right) ,\quad \forall \ell = ( p, q) \in \cotbun
.
\]
In this paper the switching time $\^\tau_1$ plays a special role,
hence we consider all the flows as starting at time $\^\tau_1$. We
denote the flow of $\vH{t}$ from time $\^\tau_1$ to time $t$ by
\[
\cH \colon (t,\ell) \mapsto \cH(t,\ell)=\cH_t(\ell) .
\]

We keep these notation throughout the paper, namely the overhead arrow
denotes the vector field associated to a Hamiltonian and the script
letter denotes its flow from time $\^\tau_1$, unless otherwise stated.

Finally recall that any vector field $f$ on $\spa$ defines, by lifting
to the cotangent bundle, a Hamiltonian
\begin{equation*}
  F \colon \ell =(p, q) \in \cotbun \mapsto \scal{p}{f(q)} \in \R.
\end{equation*}
We denote by $F_0^r$, $F_1^r$, $H_1^r$, $H_2^r$, the Hamiltonians
associated to $f_0^r$, $f_1^r$, $h_1^r$, $h_2^r$, respectively and by
\[
\Fr{i_1 i_2\dots
  i_k}:=\{\Fr{i_1},\{\dots\{\Fr{i_{k-1}},\Fr{i_k}\}\dots \},\quad
i_1,\dots , i_k\in\{0,\ 1\}
\]
the Hamiltonian associated to $f^r_{i_1 i_2\dots i_k}:=
[f^r_{i_1},[\dots[f^r_{i_{k-1}},f^r_{i_k}]\dots ],$ where
$\{\cdot,\cdot\}$ denotes the Poisson parentheses between Hamiltonians
and $[\cdot,\cdot]$ denotes the Lie brackets between vector fields.



The flow from time $\^\tau_1$ of the reference vector field $\^ f_t $
 is a map defined in a neighbourhood of
the point $ \xJ := \^\xi(\^\tau_1)$.  We denote it as $ \^ S_t \colon
\spa \to \spa, \quad t \in [0, \^ T] $ while
\[
\^{F}_t =\begin{cases}
  H_1 &\mbox{if }\ t \in [0,\^\tau_1 )\\
  F_0 + \^ u(t) F_1 &\mbox{if }\ t \in (\^\tau_1 ,\^\tau_2) \\
  H_2 &\mbox{if }\ t \in (\^\tau_2, \^T ]
\end{cases}
\]
denotes the time--dependent reference Hamiltonian obtained lifting $\^
f_t$.

Moreover we define $H^{\max, r}$ to be the continuous maximised
Hamiltonian associated to the control system
\eqref{eq:2r}--\eqref{eq:4r}, i.e.
\[H^{\max, r} \colon \ell \mapsto \max_{u\in[-1,1]}\{ F_0^r (\ell) + u
F_1^r(\ell) \}.
\]
To facilitate reading, when $r =0$ we omit the parameter, i.e.~we
write $f_0$ instead of $f_0^0$, $f_1$ instead of $f_1^0$, $H^{\max}$
instead of $H^{\max, 0}$ and so on.

Also we use the following notation from differential geometry: $f\cdot
\a$ is the Lie derivative of a function $\a$ with respect to the
vector field $f$. Moreover, if $G$ is a $C^1$ map from a manifold $X$
in a manifold $Y$, we denote its tangent map at a point $x \in X$ as
$G_* $, if the point $x$ is clear from the context.

\section{Assumptions on the nominal problem}
\label{sec:assnominal}
In this section we state the assumptions on the nominal
extremal. Besides Pontryagin Maximum Principle, we state the
assumptions which ensure strong local optimality of the reference
trajectory, see \cite{PS11}: regularity assumptions on the bang arcs
and on the junction points and a coercivity assumption of a suitable
second variation on the singular arc. We are also making one further
assumption, i.e.~controllability along the
singular arc or, equivalently the uniqueness of the adjoint covector.

\subsection{Pontryagin Maximum Principle and Regularity Assumptions}
\label{sub:PMP}
In this section we recall the first order optimality condition which
the reference triplet $(\^T, \^\xi, \^u)$ must satisfy.

We call {\em extremal pair} of $\Pbz$ any curve in the cotangent
bundle which satisfies PMP and {\em extremal trajectory} of $\Pbz$ its
projection on the state space.
Here we ask for the reference trajectory to be a {\em normal} extremal
trajectory, i.e.~we assume that the triplet $(\^T, \^\xi, \^u)$
satisfies the following
\begin{assumption}[Normal PMP]\label{ass:PMP}
  There exists a solution $\^ \lambda = \left( \^\mu, \^\xi \right)
  \colon [0, \^ T] \to \cotbun$ of the Hamiltonian system
  \begin{equation*}
    \dot \lambda (t) =\vFref{t}\* \lambda(t)
  \end{equation*}
  such that
  \begin{equation}\label{eq:max}
    \scal{\^\mu(t)}{\^ f_t \* \^\xi(t)} = \^ F_t \* \^\lambda(t) = H^{\max} \* \^\lambda(t) = 1
    \quad \qo t\in [0,\^ T].
  \end{equation}
  $\^\mu$ is called {\em nominal adjoint covector} and satisfies the
  adjoint equation
  \[
  \dot{\^\mu}(t) = - \, \dfrac{\partial \^ F_t }{\partial q}\left(
    \^\mu(t), \^\xi (t) \right).
  \]
\end{assumption}
We denote the initial point, the junction points between the bang and
the singular arcs and the final point of $\^\lambda$ as
\begin{alignat*}{2}
  & \lri = (\mri, \^x_0) :=\^\lambda(0), \quad &&
  \lrp = (\mrp, \^x_1) :=\^\lambda(\^\tau_1), \\
  & \lrs = (\mrs, \^x_2) :=\^\lambda(\^\tau_2), \quad && \lf =
  (\^\mu_f, \^x_f) := \^\lambda(\^T),
\end{alignat*}
respectively. Because of the structure of the reference control $\^
u$, as defined by equations \eqref{eq:refcontrol}, PMP implies
\begin{alignat}{2}
  & u_1 F_1\*\^ \lambda (t) \geq 0 \quad
  && t\in [0,\^\tau_1) , \label{eq:F1b1} \\
  & F_1\* \^\lambda (t) = 0 \quad
  && t \in [\^\tau_1, \^\tau_2], \label{eq:H1} \\
  & u_2 F_1\*\^ \lambda (t) \geq 0 \quad && t\in (\^\tau_2 ,\^T ] .
  \label{eq:F1b2}
\end{alignat}
As a consequence, see \cite{PS11}, one gets
\begin{alignat}{2}
  & F_{01}\*\^\lambda (t) \equiv 0 \quad t \in [\^\tau_1,
  \^\tau_2], \label{eq:H01}\qquad && \left(F_{001} + \^
    u(t)F_{101}\right)\*\^\lambda(t) = 0 \quad t\in (\^\tau_1,
  \^\tau_2 ), 
  \\
  & u_1 \left( F_{001} + u_1 F_{101} \right) (\^\ell_1) \geq
  0, \label{eq:BL} && u_2 \left( F_{001} + u_2 F_{101} \right)
  (\^\ell_2) \geq 0.
\end{alignat}
PMP yields the mild inequalities in \eqref{eq:F1b1}, \eqref{eq:F1b2}
and \eqref{eq:BL}. We assume the strict inequalities to hold, whenever
possible.
\begin{assumption}[Regularity along the bang arcs] \label{ass:regbang}
  \[
  u_1 F_1\*\^\lambda(t) > 0 \quad \forall t\in [0,\^\tau_1), \qquad
  u_2 F_1\*\^\lambda(t) > 0 \quad \forall t\in (\^\tau_2, \^T].
  \]
\end{assumption}
\begin{assumption}[Regularity at the junction
  points]\label{ass:regjun}
  \begin{equation*}
    \left( u_1 F_{001} + F_{101} \right)(\lrp) > 0 , \qquad
    \left( u_2 F_{001} + F_{101} \right)(\lrs) > 0 .
  \end{equation*}
\end{assumption}
Another well known necessary condition for the local optimality of a
Pontryagin extremal is the {\em generalised Legendre condition} (GLC)
along the singular arc:
\begin{equation*}
  F_{101} \* \^ \lambda (t) \geq 0 \quad t\in
  [\^\tau_1, \^\tau_2 ]\, , 
\end{equation*}
see for example \cite{AS04}, Corollary 20.18 page 318; for a classical
result see \cite{GK72}.
The coercivity assumption stated in the next section implies the {\em
  Strengthened generalised Legendre condition}
\begin{equation} \label{eq:SGLC} %
  \tag{SGLC} F_{101}\*\^\lambda (t) > 0 , \quad t \in [\^\tau_1,
  \^\tau_2].
\end{equation}
When \eqref{eq:SGLC} holds, a singular extremal is called {\em of the
  first kind}, see e.g.~\cite{ZB94}.
\begin{remark}\label{re:refcont}
  \eqref{eq:SGLC} implies that $\^ u\in C^\infty
  ((\^\tau_1,\^\tau_2))$ and that Assumption \ref{ass:regjun} is
  equivalent to the discontinuity of $\^ u$ at times $\^\tau_1$ and
  $\^\tau_2$, see \cite{PS11}.
\end{remark}
 \begin{assumption}[Uniqueness of the adjoint covector]\label{ass:uniq}
   $\left. \^\lambda\right\vert_{[\^\tau_1, \^\tau_2]}$ is the only
   adjoint covector associated to $\left. \^\xi\right\vert_{[\^\tau_1,
     \^\tau_2]}$ for the minimum time problem between
   $\^\xi(\^\tau_1)$ and $\^\xi(\^\tau_2)$.
 \end{assumption}

 \subsection{Coercivity and controllability
   assumptions}\label{sec:ejse}
 System \eqref{eq:2r} is affine with respect to the control, therefore
 the standard second variation is completely degenerate. In
 \cite{PS11} we transformed the given minimum time problem in a Mayer
 problem on a fixed time interval and -- via a coordinate-free version
 of Goh's transformation -- we obtained a suitable second order
 approximation on the singular arc, which we call {\em extended second
   variation}.

 Proceeding as in Lemma 1 of \cite{PS12} one can show that the largest
 sub--space where the extended second variation can be coercive is the
 one relative to the minimum time problem with fixed end points
 $\xi(\^\tau_1) = \xJ$, $\xi(\^\tau_2)=\^x_2 $.

 We point out that the same assumption, together with Assumptions
 \ref{ass:regbang}--\ref{ass:regjun} is sufficient for $\^\xi$ to be a
 minimum time trajectory between $\^x_0$ and $\^x_f$, see \cite{PS11}.

 For the sake of completeness we write here the above mentioned Mayer
 problem:
 \begin{equation*}
   \text{Minimise}\quad \xi^0(\^\tau_2)
 \end{equation*}
 subject to
 \begin{align*}
   \begin{split}
     & \dot{\xi}^0 (s) = u_0(s)  \\
     & \dot{\xi}(s) = u_0(s) \, f_0(\xi(s)) + u_0(s)\, u(s) \,
     f_1(\xi(s)) \\
     & (u_0(s), u(s)) \in (0, +\infty) \times (-1, 1)
   \end{split}
   \qquad s \in [\^\tau_1, \^\tau_2]     \\
   & \xi^0(\^\tau_1) = \^\tau_1, \quad \xi(\^\tau_1) = \xJ \, , \quad
   \xi^0(\^\tau_2) \in \R , \quad \xi(\^\tau_2) = \^ x_2 .
 \end{align*}
 Also, for the sake of future computations we introduce the dragged
 vector fields at time $\^\tau_1$, along the reference flow, by
 setting
 \begin{equation*}
   \gi{i}{t}( x) := \^{ S}_{t*}\inv f_i \* \^{ S}_t( x)\, , \ i=0, 1,
   \quad
   \quad 
   \^g_t := \^{ S}_{t*}\inv \^f_t \* \^{ S}_t( x) = 
   \gi{0}{t} + \^u(t)\gi{1}{t},
 \end{equation*}
 and we recall that
 \begin{equation*}
   \dgi{1}{t}(x) =\^{ S}_{t*}\inv f_{01} \* \^{
     S}_t( x) , \qquad \dgi{0}{t}(x) =
   - \^ u(t) \dgi{1}{t}(x). 
 \end{equation*}

 Since the extremal $\^\lambda$ is normal, $f_0 $ and $f_1 $ are
 linearly independent at $\xJ$, so that we may choose local
 coordinates around $\xJ$ which simplify computations. Namely, we
 choose coordinates $y =
 \begin{pmatrix}
   y_1, \ldots, y_n
 \end{pmatrix}
 $such that
 \begin{equation*}
   \begin{minipage}{0.85\textwidth}
     \begin{packedenuma}
     \item $f_1$ is constant: $f_1 \equiv \direz{y_1}$,
     \item $f_0 = \direz{y_2} - y_1 \left( f_{01}(\xJ) + O(y)
       \right)$.
     \end{packedenuma}
   \end{minipage}
 \end{equation*}
 In such coordinates choose $\beta$ as $
 \beta(y) := - \dsum_{i=2}^{n} \mu_i y_i, $ where $(0, \mu_2, \ldots,
 \mu_n)$ are the coordinates of $\mrp$.  We get $\mu_2 = 1$, $f_1
 \cdot \beta \equiv 0$, and $f_0 \cdot f_0 \cdot \beta(\xJ) = 0$. In
 these coordinates the extended second
 variation 
 is thus actually given by the quadratic form
 \begin{multline} \label{eq:intsolo} \qquad J\se_\ext (\e_0, \e_1 , w)
   = \dfrac{1}{2}\int_{\^\tau_1}^{\^\tau_2} \big( w^2(t)[\dgi{1}{t},
     \gi{1}{t}]\cdot \beta(\xJ) + \\ 
     + 2 w(t)\,
     \zeta(t) \cdot \dgi{1}{t} \cdot \beta(\xJ) \big) \ud t \qquad
 \end{multline}
 defined on the linear sub--space $\cW$ of $\R^2 \times L^2([\^\tau_1,
 \^\tau_2], \R)$ of the triplets $(\e_0, \e_1, w)$ such that the
 linear system
 \begin{equation}\label{eq:sistjsefinal}
   \dot\zeta(t) = w(t) \dgi{1}{t}(\xJ), \quad
   \zeta(\^\tau_1) =\e_0 f_0 (\xJ) + \e_1 \, f_1 (\xJ), \quad 
   \zeta(\^\tau_2) = 0
 \end{equation}
 admits a solution $\zeta$, see \cite{PS11}.
 \begin{assumption}[Coercivity] \label{ass:coerc} The extended second
   variation for the minimum time problem between fixed end points on
   the singular arc is coercive. Namely we require that the quadratic
   form \eqref{eq:intsolo} is coercive on the subspace $\cW$ of
   $\R^2\times L^2([\^\tau_1, \^\tau_2], \R)$ given by the variations
   $\delta e=(\e_0,\e_1,w) $ such that system \eqref{eq:sistjsefinal}
   admits a solution.
 \end{assumption}
 \begin{remark}\label{re:jsec}
   \begin{packedenum}
   \item $J\se_\ext$ is a quadratic form defined in the whole space
     $\R^2 \times L^2([\^\tau_1, \^\tau_2], \R)$, but only its
     restriction to $\cW$ is coordinate free.
   \item Notice that
     \[
     R(t) := [\dgi{1}{t}, \gi{1}{t}]\cdot \beta(\xJ) =
     F_{101}(\^\lambda(t)) > 0.
     \]
   \item Under \eqref{eq:SGLC} $J\se_\ext$ can be proven to be the
     standard second variation, along the extremal pair
     $\left. \^\lambda\right\vert_{ [\^\tau_1, \^\tau_2]} $ of a
     nonsingular Mayer problem, see \cite{PS08} and \cite{Ste07} for
     more details.
   \end{packedenum}
 \end{remark}

 We now exploit Assumption \ref{ass:uniq} in relation to the
 controllability space (see e.g.~\cite{Conti76}) of system
 \eqref{eq:sistjsefinal}:
 \begin{equation}
   \CS := \Span \left\{f_0(\xJ),\; f_1(\xJ),\; \dgi{1}{t}(\xJ), \;
     t\in [\^\tau_1,\^\tau_2]\right\}. \label{eq:moltuni2}
 \end{equation}
 \begin{lemma}\label{le:unico}
   Assumption \ref{ass:uniq} holds if and only if $\CS = \tansp$.
 \end{lemma}
 \begin{proof}{\itshape $\CS = \tansp$ implies Assumption
     \ref{ass:uniq}. }
   Assume by contradiction that there exists a different adjoint
   covector $\mu(t) = \^\mu(t) + \omega(t) = (\mrp + \omega_1)\^S_{t\,
     *}\inv$ with associated multiplier $\pi_0 \in \{0, 1\}$. By
   \eqref{eq:max}
   \begin{equation}
     \label{eq:perp4}
     \scal{\omega_1}{\gi{1}{t}(\xJ)} = 0 \, , \qquad
     \scal{\omega_1}{\gi{0}{t}(\xJ)} = \pi_0 - 1
   \end{equation}
   which, for $\pi_0 = 1$ yield
   \begin{equation}
     \label{eq:perp5}
     \scal{\omega_1}{\dgi{1}{t}(\xJ)} = 0 \, , \qquad
     \scal{\omega_1}{f_1(\xJ)} = 0 \, , \qquad
     \scal{\omega_1}{f_0(\xJ)} = 0
   \end{equation}
   that is $\omega_1 \in \CS^\perp = \{0\}$.

   If $\pi_0 = 0$, then $\mrp + \omega \in \CS^\perp = \{0\}$ so that
   the new multiplier is the trivial one, a contradiction.

   \noindent{\itshape Assumption \ref{ass:uniq} implies $V =\tansp$.
   }
 %
   Suppose, by contradiction, that there exists $\omega \neq 0$,
   $\omega \in \CS^\perp$ so that
   \begin{equation*}
     \begin{split}
       & \scal{\omega}{\gi{1}{t}(\xJ)} = \scal{\omega}{f_1(\xJ) } +
       \int_{\^\tau_1}^t \scal{\omega}{\dgi{1}{s}(\xJ)} \ud s = 0, \\
       & \scal{\omega}{\gi{0}{t}(\xJ)} = \scal{\omega}{f_0 (\xJ) +
         \int_{\^\tau_1}^t \dgi{0}{s}(\xJ) \ud s} = -
       \int_{\^\tau_1}^t \^u(s) \scal{\omega}{\dgi{1}{s}(\xJ)} \ud s =
       0.
     \end{split}
   \end{equation*}
   Therefore $(\mrp+\omega)\^S_{t\, *}\inv$ is an adjoint covector
   along the singular arc of $\^\xi$ with multiplier $p_0 = 1$, a
   contradiction.
 \end{proof}

 \subsection{Consequences of coercivity and controllability}
 \label{sec:consequences}

 In order to exploit the coercivity assumption we follow \cite{SZ97}
 and we introduce the Lagrangian subspace and the Hamiltonian
 associated to the second variation \eqref{eq:intsolo},
 \eqref{eq:sistjsefinal}, respectively given by
 \begin{align}
   \begin{split}
     & L := \left\{ f_0(\xJ), \ f_1(\xJ) \right\}^\perp \times
     \Span\left\{ f_0(\xJ), \ f_1(\xJ)
     \right\} = \\
     & \phantom{L :}=  \R \vF{0}(\lrp) \oplus \R \vF{1}(\lrp) \oplus 
     \left( 
\{ f_0(\^ x_1),
     f_1(\^ x_1) \}^\perp  \times \{ 0 \}
\right),
   \end{split}\label{eq:lse} \\
   & H\se_t (\omega, \dx) := \dfrac{-1}{2 R(t)} \left(
     \scal{\omega}{\dgi{1}{t}(\xJ)} + \dx \cdot \dgi{1}{t}\cdot
     \beta(\xJ) \right)^2. \label{eq:Hse}
 \end{align}
 \begin{lemma}\label{le:coerc}
   Let $\cH\se_t \colon \cotsp \times \tansp \to \cotsp \times \tansp$
   be the flow of the Hamiltonian $H\se_t$ defined in \eqref{eq:Hse}.
   Under Assumptions \ref{ass:uniq} and \ref{ass:coerc} the kernel of
   the linear mapping $ \left. \pi_* \cH^{\se}_{\^\tau_2}
   \right\vert_{L} $ is trivial.
 \end{lemma}
 \begin{proof}
   It is an easy consequence of the results in \cite{SZ97} that the
   quadratic form $J\se_{\ext}$ is coercive if and only if for all
   $(\omega, \dx) \in L$ and all $t \in [\^\tau_1, \^\tau_2]$
   \begin{equation}
     \label{eq:coerc}
     \pi_*\cH\se_t(\omega, \dx) = 0 \quad \text{ implies } \quad \begin{cases}\dx =0 \\
       \cH\se_s(\omega,0) = (\omega, 0) \quad \forall s \in [\^\tau_1, t].
     \end{cases}
   \end{equation}
   Ler $(\omega, \dx) \in \ker \left. \pi_* \cH^{\se}_{\^\tau_2}
   \right\vert_{L} $. By \eqref{eq:coerc} $\dx = 0$, $\omega \in
   \left\{ f_0(\^ x_1), f_1 (\^ x_1)\right\}^\perp$ and $\left(
     \mu(t), \zeta(t)\right) := \cH^{\se}_t (\omega, \dx ) = (\omega,
   0)$ for any $t \in [\^\tau_1, \^\tau_2]$. Since the equations for
   $\left( \mu(t), \zeta(t) \right) $ are
   \begin{align}
     & \dot\mu(t) = \dfrac{1}{R(t)} \Big(
     \scal{\mu(t)}{\dgi{1}{t}(\xJ)} + \zeta(t) \cdot \dgi{1}{t}\cdot
     \beta(\xJ)
     \Big) \left( \cdot \right) \cdot \dgi{1}{t} \cdot \beta(\xJ) \\
     & \dot\zeta(t) = \dfrac{-1}{R(t)} \Big(
     \scal{\mu(t)}{\dgi{1}{t}(\xJ)} + \zeta(t) \cdot \dgi{1}{t}\cdot
     \beta(\xJ) \Big) \dgi{1}{t} (\xJ) .
   \end{align}
   we get $\scal{\omega}{\dot g^1_t(\^x_1)} = 0$ for all $t \in [\^\tau_1, \^\tau_2]$. Thus, Assumption
   \ref{ass:uniq} yields the claim.
 \end{proof}

 \section{The main results}
 In this Section we state the main results of the paper, Theorem
 \ref{thm:main} and \ref{thm:uniq}, which will be proven in the
 following Section, and provide an example.
 \label{sec:main}
 \begin{theorem}
   \label{thm:main}
   Under Assumptions \ref{ass:PMP}--\ref{ass:coerc}, there exists
   $\raggio > 0$ such that for any $r$, $\norm{r}< \raggio$, problem $\Pbr$
   has a bang-singular-bang strict (time, state)-local optimiser
   $\xi^{r}$. The switching times and the final time of $\xi^r$ depend
   smoothly on $r$.  If $\^{\xi}$ is injective, then $\xi^r$ is a
   state-local optimiser of $\Pbr$.
 \end{theorem}

 First we prove the existence of the bang-singular-bang extremal
 trajectory $\xi^r$, by Hamiltonian methods and the implicit function
 theorem (see Lemma \ref{le:fond}). Then the optimality of $\xi^r$ is
 proven by showing -- via standard methods of functional analysis --
 that the coercivity and the injectivity conditions are stable under
 small perturbations of the parameter $r$, see Lemmata
 \ref{le:varsecr} and \ref{le:simpler}.

 We point out that using the implicit function theorem allows to
 perform a sensitivity analysis in a standard way; this will be the
 object of a future analysis.

 Furthermore we prove the uniqueness of the extremal pair $\lambda^r =
 \left( \mu^r, \xi^r \right)$, defined in Theorem \ref{thm:main}, in a
 suitable neighbourhood of the graph of the nominal pair $\^\lambda$.

\begin{theorem}\label{thm:uniq}
  Under Assumptions \ref{ass:PMP}--\ref{ass:coerc}, there exist $\raggio > 0
  $, $\e > 0$ and a neighbourhood $\cV$ of the graph of $\^\lambda$ in
  $\R \times \fibcot$ such that for any $r$, $\norm{r} < \raggio$, the
  extremal pair $\lambda^r$ associated to the local optimiser $\xi^r$
  of Theorem \ref{thm:main} is the only extremal pair of $\Pbr$ whose
  graph is in $\cV$ and whose final time is in $[\^T - \e , \^T +
  \e]$.
\end{theorem}
The proof of this result is quite technical and is given in Section
\ref{sec:pfuniq}, we conclude this section with an example.

\subsection{Dubins car}
\label{sec:dubins}
A classical minimum time problem is the so-called Dubins car problem,
where the dynamics describes the motion of a car moving in a plane
with fixed speed and with bounded, controlled angular velocity. The
car has to be steered from a given initial position $(x_0, y_0)$ and
orientation $\theta_0$ to a prescribed final position $(x_f, y_f)$ and
orientation $\theta_f$. Namely the problem is
\begin{equation}
  \label{eq:dubins}
  \begin{split}
    &\text{minimise } \ T \ \text{ subject to} \\
    & \dot x(t) = \cos\theta(t), \quad \dot y(t) = \sin \theta(t),
    \quad
    \dot \theta(t) = u(t) , \\
    & (x(0), y(0), \theta(0)) = (x_0, y_0, \theta_0), \quad (x(T),
    y(T), \theta(T)) = (x_f, y_f, \theta_f), \\
    & \abs{u(t)} \leq 1 .
  \end{split}
\end{equation}
It can be proven that the only singular control is $u \equiv 0$ and
that, if the initial and final positions on the $(x, y)$-plane are
sufficiently far, then the optimal trajectory is bang-singular-bang,
see e.g.~\cite{AS04}.  This example fits our assumptions with $f_0(x,
y, \theta) = \left(\cos\theta, \sin\theta, 0\right)^t$ and $f_1(x, y,
\theta) = \left( 0,0,1 \right)^t$. An easy computation shows that both
Assumptions \ref{ass:regbang} and \ref{ass:regjun} are satisfied. In
\cite{PS08} it is shown that the second variation associated to any
singular trajectory between two fixed end points is
coercive. Moreover, since $\Span\{ f_0, f_1, f_{01} \}(x, y, \theta) =
\R^3$ for any $ (x, y, \theta) \in \R^3$, also Assumption
\ref{ass:uniq} is trivially satisfied.  Thus the bang-singular-bang
structure of optimisers in the Dubins car problem is stable under
small perturbations of the data of the problem.

When the final orientation $\theta(T)$ is not prescribed, the problem
is also quoted as Dodgem car problem, see e.g.~\cite{Cra95}.  In this
case when the initial and final positions on the $(x, y)-$plane are
sufficiently far, optimal trajectories are the concatenation of a bang
and of a singular arc. The same assumptions stated here for
bang-singular-bang extremals yield both optimality and stability of
such trajectories, provided that the perturbed final constraint is an integral line
of the perturbed controlled vector field.
Some
preliminary results are in \cite{PS08} and \cite{PS12}. Complete
proofs will appear in \cite{PS13}.

\section{Proof of the results}
\label{sec:proof}

\subsection{Hamiltonian approach}\label{sub:hami}
In this section we describe some properties of the Hamiltonians linked
to our system near the singular arc of the reference extremal, for
more details see \cite{PS11}.

By \eqref{eq:H1}, \eqref{eq:H01} and \eqref{eq:SGLC}, any singular
extremal of the first kind of $\Pbz$ belongs to the set
\[
\cS := \{\ell \in T^*M \colon F_1(\ell)= F_{01}(\ell)= 0 ,
F_{101}(\ell)>0 \},
\]
a subset of $ \Sigma := \{\ell\in T^*M \colon F_1(\ell)=0 \}$, where
the maximised Hamiltonian of $\Pbz$, $H^{\max}$, coincides with every
Hamiltonian $F_0 + u F_1$, $u \in \R$.

Notice that $\cS$ and $\Sigma$ are independent of the control
constraints but, by \eqref{eq:4r}, \eqref{eq:H01} and Remark
\ref{re:refcont}, any singular extremal of problem $\Pbz$ is in
\[
\cS\cap \left\{\ell \in T^*M \colon \abs{ \frac{
      F_{001}}{F_{101}}(\ell) } < 1 \right \}.
\]
The following results are proven in Lemmata 2 and 3 of \cite{PS11}:
\begin{lemma}\label{le:geo}
  If \eqref{eq:SGLC} holds, then there exists a neighbourhood $\cV$ of
  $\cS$ in $\cotbun$ where the following statements hold true.
  \begin{enumerate}
  \item $\Sigma \cap \cV$ is a hyper--surface and $\cS \cap \cV$ is a
    $(2n - 2)$-dimensional symplectic manifold.  Moreover $\Sigma$
    separates the regions defined by: $H^{\max}=F_0 + F_1$, $H^{\max}
    = F_0 - F_1$.
  \item The Hamiltonian vector field $\vF{1}$ is tangent to $\Sigma$
    and transverse to $\cS$, while $\vF{01}$ is transverse to
    $\Sigma$.
  \item Setting $v := \dfrac{-F_{001}}{F_{101}}$ we obtain {\it the
      Hamiltonian of singular extremals of the first kind}
    \[
    F^{S} := F_0 + v \,F_1 ,
    \]
    i.e.~the associated vector field $\vFS$ is tangent to $\cS$ and
    any singular extremal of the first kind of $\Pbz$ is an integral
    curve of $\vFS$ contained in $\cS$.
  \item There exists a non-negative smooth Hamiltonian $\chi \colon
    \cV \to [0, +\infty)$ such that
    \begin{enumerate}
    \item $\chi = 0$, $\vchi = 0 \ $ and $\ \uD^2 \chi =
      \dfrac{1}{F_{101}}\uD F_{01} \otimes \uD F_{01} \ $ on $\ \cS$;
    \item $\vF{0} + \vchi$ is tangent to $\Sigma$.
    \end{enumerate}
  \end{enumerate}
\end{lemma}
\noindent From now on we shall denote $\Sigma \cap \cV$ and $\cS \cap
\cV$ as $\Sigma$ and $\cS$, respectively.
 
\noindent Since for the nominal problem \eqref{eq:SGLC} holds true in
the neighbourhood $\cV$ of $\left. \^\lambda \right\vert_{[\^\tau_1,
  \^\tau_2]}$ defined in Lemma \ref{le:geo}, then possibly restricting
$\cV$ and for small enough $\norm{r}$, \eqref{eq:SGLC} holds also for
the Hamiltonians $F_{101}^r$. Therefore we can define, in $\cV$, the
Hamiltonians of singular extremals of $\Pbr$
\begin{equation*}
  F^{S, r} := F_0^r - \dfrac{F_{001}^r}{F_{101}^r}F_1^r .
\end{equation*}
 
In order to prove our main result we are going to use the following
result from \cite{PS11}.
\begin{lemma}\label{le:trasporti}
  If \eqref{eq:SGLC} holds, then the Hamiltonian vector field
  $\overrightarrow{\^H}_t := \overrightarrow{\^F}_t + \vchi$ is
  tangent to $\Sigma$. For any $t \in \left[ \^\tau_1, \^\tau_2
  \right]$ the derivative of its flow $\^\cH_t$ satisfies the
  following properties:
  \begin{enumerate}
  \item $\^\cH_{t*}\vF{1}(\lrp) = \vF{1}(\^\lambda(t))$ and
    $\^\cH_{t*}\vF{0}(\lrp) = \vF{0}(\^\lambda(t))$
  \item If $\dl_S \in T_{\lrp}\cS$ then
    \[
    \cF^S_{t*}\dl_S = \^\cH_{t*}\dl_S + a(t, \dl_S)
    \vF{1}(\^\lambda(t))
    \]
    where $\cF^S_t$ is the flow of $\vFS$ and $ a(t, \dl_S) :=
    \displaystyle\int_0^t\scal{\uD
      v(\^\lambda(s))}{\cF^S_{t*}\dl_S}\ud s $.
  \end{enumerate}
\end{lemma}
\begin{proof}
  Claim 1 is proven in Lemma 4 of \cite{PS11}. \\
  {\em Proof of Claim 2: } The flow $\cG_t :=
  \^\cH_t^{-1}\circ\cF_t^S$ is the Hamiltonian flow associated to $G_t
  := \left( (v - \^u(t)) F_1 - \chi\right) \*\^\cH_t $. Since $\uD
  G_t(\lrp) = 0$, then $\cG_{t*} = \^\cH_{t*}^{-1} \cF^S_{t*}$ is the
  linear flow associated to the quadratic Hamiltonian
  \[
  \uD^2 G_t(\lrp) = \left( \uD v \otimes \uD F_1 + \uD F_1 \otimes \uD
    v - \dfrac{1}{F_{101}} \uD F_{01} \otimes \uD F_{01}
  \right)(\^\lambda(t)) \^\cH_{t*} \otimes \^\cH_{t*} .
  \]
  Set $\gamma(t) := \cG_{t*}\dl$. Since $\^\cH_{t*}\cG_{t*}\dl =
  \cF^S_{t*}\dl \in T_{\^\lambda(t)}\cS$, we obtain, by Claim 1 that
  $\dot\gamma(t) = \scal{\uD
    v(\^\lambda(t))}{\^\cF_{t*}\dl}\vF{1}(\lrp)$. Thus $\cG_{t*}\dl =
  \gamma(t) = \dl + \int_0^t\scal{\uD
    v(\^\lambda(s))}{\cF^S_{t*}\dl}\ud s \, \vF{1}(\lrp)$ which,
  together with Claim 1, completes the proof.
\end{proof}

We end this section by rephrasing Lemma \ref{le:coerc} in terms of the
flow $\^\cH$ defined in Lemma \ref{le:trasporti}.  This is done
adapting the proof of Claim 1 in Lemma 9 of \cite{PS11}.
\begin{corollary}\label{cor:proie}
  Under Assumptions \ref{ass:uniq} and \ref{ass:coerc} the kernel of
  the linear map $\pi_* \^\cH_{\^\tau_2 \, *} \colon L \to \tansp$ is
  trivial.
\end{corollary}
\begin{proof}
  $\cG_t := \^\cF\inv_{t*}\^\cH_{t*}$ is the linear flow associated to
  the quadratic Hamiltonian
  \[
  \begin{split}
    G_t(\omega, \dx) & = \dfrac{1}{2R(t)}\left( \uD
      F_{01}(\^\lambda(t))\left( \omega \^S\inv_{t*} ,
        \^S_{t*}\dx\right)
    \right)^2 \\
    & = \dfrac{1}{2R(t)}\left( \scal{\omega\^S\inv_{t*}
      }{f_{01}(\^\xi(t))} + \scal{\^\mu(t)}{\uD f_{01}(\^\xi(t))
        \^S_{t*}\dx}
    \right)^2 \\
    & = \dfrac{1}{2R(t)}\left( \scal{\omega}{\dot g_{1, t}(\^ x_1)} -
      \dx \cdot \dot g_{1, t} \cdot \beta(\^x_1) \right)^2.
  \end{split}
  \]
  Consider the linear isomorphism $i \colon (\omega, \dx) \mapsto
  (-\omega, \dx)$. Then $ G_t = - H\se_t \* i $ and $ \vG_{t} = i \*
  \overrightarrow{H\se_t} \* i$ so that $ \pi_* \cH\se_t i =
  \pi_*\cG_t = \pi_* \^\cF\inv_{t*} \cH_{t*} = \^S\inv_{t*} \pi_*
  \^\cH_{t*}.  $ Since $iL = L$, from Lemma \ref{le:coerc} we finally
  get the claim.
\end{proof}

\subsection{Existence of an extremal}
\label{sec:fundlemma}
In the following lemma we prove the existence of a bang-singular-bang
extremal for $\Pbr$.
\begin{lemma}
  \label{le:fond}
  There exist $\raggio > 0 $, $\e > 0$ and a neighbourhood $\cO$ of
  $\mri$ in $\cotsp$ such that for any $r$, $\norm{r} < \raggio$ there
  exists a unique normal bang--singular--bang extremal par $\lambda^r
  := \left( \mu^r, \xi^r \right)$ of $\Pbr$ with the following
  properties
  \begin{packedenum}
  \item $\mu^r(0) \in \cO $;
  \item the first switching time $\tau_1(r)$ is in $[\^\tau_1 - \e,
    \^\tau_1 + \e ]$;
  \item the second switching time $\tau_2(r)$ is in $[\^\tau_2 - \e,
    \^\tau_2 + \e ]$;
  \item the final time $T(r) $ is in $[\^T - \e, \^T + \e ]$.
  \item the times $\tau_1(r)$, $\tau_2(r)$ and $T(r)$ and the initial
    adjoint covector $\omega(r) := \mu^r(0)$ depend smoothly on $r$.
  \end{packedenum}
  Moreover the bang arcs are regular
  \[
  u_1 F_1^r \*\lambda^r(t) > 0 \quad \forall t \in [0, \tau_1^r) ,
  \qquad u_2 F_1^r \*\lambda^r(t) > 0 \quad \forall t \in (\tau_2^r,
  \^T] ,
  \]
  and the singular arc is of the first kind
  \[
  F_{101}^r\*\lambda^r(t) > 0 \quad \forall t \in [\tau_1^r ,\tau_2^r
  ] .
  \]
\end{lemma}
\begin{proof}
  The proof of the lemma is a straightforward application of the
  implicit function theorem.  Let $B(0,\raggio)$ be the ball of radius
  $\raggio > 0$ centred at the origin in $\R^m$. If $\raggio $ and the
  neighbourhood $\cO$ are sufficiently small, we can define the
  following map
  \begin{multline}
    \label{eq:Phi}
    \qquad \Phi \colon (r, \omega, \tau_1, \tau_2, T) \in B(0, \raggio)
    \times \cotsp
    \times \R^3 \mapsto \\
    \pi\exp (T - \tau_2) \vHr{2} \* \exp(\tau_2 - \tau_1)\vFSr \*\exp
    \tau_1\vHr{1}(\omega, a^r) - b^r \in \spa . \qquad
  \end{multline}
  Let
  \begin{multline}\label{eq:impl}
    \qquad \Psi(r, \omega, \tau_1, \tau_2, T) = \left( \Phi (r,
      \omega, \tau_1, \tau_2, T) ,
      \Fr{1} \* \exp \tau_1\vHr{1}(\omega, a^r) , \right. \\
    \left. \Fr{01} \* \exp \tau_1\vHr{1}(\omega, a^r) , \Fr{0} \* \exp
      \tau_1\vHr{1}(\omega, a^r) - 1 \right), \qquad
  \end{multline}
  we prove that the Jacobian matrix $\left. \dfrac{\partial \Psi
    }{\partial(\omega, \tau_1, \tau_2, T)}\right\vert_{(0, \lri,
    \^\tau_1, \^\tau_2, \^T)}$ is non-degenerate, so that the implicit
  equation $\Psi
  \begin{pmatrix}
    r, \omega, \tau_1, \tau_2, T
  \end{pmatrix} = 0 $ defines smooth functions
  \[
  \omega(r), \ \tau_1(r), \  \tau_2(r),   \    T(r), \ \norm{r} < \raggio
  \]
  for some positive $\raggio$. Indeed, the matrix is equal to
  \begingroup
  \everymath{\scriptstyle}
  \begin{equation*}
    \begin{pmatrix}
      \exp(\^T - \^\tau_2)h_{2 *}\pi_*\cFS_{\^\tau_2 *}\!
      \exp\^\tau_1\vH{1 }_*(\cdot,
      0) 
      \!\!\!\!& \left. c_1 \! \exp(\^T - \^\tau_2)h_{2*} \pi_*
        \cFS_{\^\tau_2 \! *}
        \vF{1}\right\vert_{\lrp} 
      \!\!\!\!& \left. - c_2 \! \exp(\^T - \^\tau_2)h_{2*}
        f_1\right\vert_{\^x_2 }
      \!\!\!\!& h_2(\^x_f)\\[5mm]
      \dueforma{\exp\^\tau_1\vH{1}_*(\cdot,
        0)}{\left. \vF{1}\right\vert_{\lrp}}
      & 0 & 0 & 0\\[5mm]
      \dueforma{\exp\^\tau_1\vH{1}_*(\cdot,
        0)}{\left. \vF{01}\right\vert_{\lrp}} & c_1 \left. F_{101} \right\vert_{\lrp}& 0 & 0 \\[5mm]
      \dueforma{\exp\^\tau_1\vH{1}_*(\cdot,
        0)}{\left. \vF{0}\right\vert_{\lrp}} & 0 & 0 & 0
    \end{pmatrix}
  \end{equation*}
  \endgroup
  where $c_1 := u_1 + \frac{F_{001}}{F_{101}}(\lrp)$ and $c_2 := u_2 +
  \frac{F_{001}}{F_{101}}(\lrs)$ are nonzero (see Remark
  \ref{re:jsec}).

  Since $\exp\^\tau_1\vH{1}_*$ is a linear isomorphism between
  vertical fibers, this matrix 
  is singular if and only if there exist $\dl := (\omega, 0)$,
  $\dtau_1$, $\dtau_2$ and $\dT$, with at least one of them different
  from zero, such that
  \begin{align}
    \label{eq:impl2a} & \pi_*\cFS_{\^\tau_2 \, *} \left( \dl + \dtau_1
      \, c_1 \, \vF{1}(\lrp) \right) - \dtau_2 \, c_2 \, f_1(\^x_2) +
    \dT \, h_2(\^x_2) = 0
    \\
    & \dueforma{\dl}{\vF{1}(\lrp)} = 0 \label{eq:impl2b}\\
    & \dueforma{\dl}{\vF{01}(\lrp)} + \dtau_1 \, c_1 \, F_{101}( \lrp ) = 0 \label{eq:impl2c}\\
    & \dueforma{\dl}{\vF{0}(\lrp)} = 0 \label{eq:impl2d}
  \end{align}
  Equation \eqref{eq:impl2c} yields $\dtau_1 = \frac{ -
    \dueforma{\dl}{\vF{01}(\lrp)}}{c_1 \, F_{101}( \lrp )} $, hence
  \[
  \dl_{\cS} := \dl + \dtau_1 \, c_1 \, \vF{1}(\lrp) = \dl -
  \frac{\dueforma{\dl}{\vF{01}(\lrp)}}{F_{101}( \lrp )} \vF{1}(\lrp)
  \in T_{\lrp}\cS,
  \]
  so that by Claims 1. and 3. in Lemma \ref{le:trasporti}
  \[
  \cFS_{\^\tau_2 \, *} \left( \dl + \dtau_1 \: c_1 \, \vF{1}(\lrp)
  \right) = %
  \^\cH_{\^\tau_2 \, *} \left( \dl -
    \frac{\dueforma{\dl}{\vF{01}(\lrp)}}{F_{101}( \lrp )} \vF{1}(\lrp)
    +a(\^\tau_2, \dl_S) \vF{1}(\lrp) \right)
  \]
  and equation \eqref{eq:impl2a} reads
  \begin{equation}
    \label{eq:impl3a}
    \begin{split}
      \pi_*\cH_{\^\tau_2\, *}\Big( \dl - &
      \frac{\dueforma{\dl}{\vF{01}(\lrp)}}{F_{101}( \lrp )}
      \vF{1}(\lrp) +a(\^\tau_2, \dl_S) \vF{1}(\lrp)
      - \\
      & - \, \dtau_2 \, c_2 \, \vF{1}(\lrp) + \delta T \, \left(\vF{0}
        + u_2 \vF{1}\right) (\lrp) \Big) = 0
    \end{split}
  \end{equation} 
  Equations \eqref{eq:impl2b} and \eqref{eq:impl2d} yield $\dl \in
  \Span\{ f_0(\xJ), f_1(\xJ)\}^\perp \times \{ 0 \} \subset L$. Thus
  Corollary \ref{cor:proie} and equation \eqref{eq:impl3a} yield
  \begin{equation}
    \label{eq:impl3}
    \dl - \left( 
      \frac{\dueforma{\dl}{\vF{01}(\lrp)}}{F_{101}( \lrp )} 
      - a(\^\tau_2, \dl_S)
      + \, \dtau_2 \, c_2 
      - u_2 \, \delta T 
    \right) 
    \vF{1}(\lrp)
    + \delta T \, \vF{0}(\lrp) = 0 .
  \end{equation}
  Since $\dl$, $\vF{0}(\lrp)$ and $\vF{1}(\lrp)$ are linearly
  independent, equation \eqref{eq:impl3} gives
  \[
  \dT = 0, \qquad \dl = \dl_S = 0, \qquad \dtau_2 = 0.
  \]
  Finally, substituting in \eqref{eq:impl2c}, we get $\dtau_1 = 0$
  which proves our claim, i.e.~
  \begin{equation}
    \label{eq:lambdar}
    \lambda^r \colon t \mapsto
    \begin{cases}
      \exp t\vHr{1}(\omega,a^r) \quad & t \in [0, \tau_1(r)] \\
      \exp(t - \tau_1(r))\vFSr \*\lambda^r(\tau_1(r)) \quad
      & t \in (\tau_1(r), \tau_2(r)] \\
      \exp (T(r) - \tau_1(r))\vHr{2}\*\lambda^r(\tau_2(r)) \quad & t
      \in (\tau_2(r), T(r)]
    \end{cases}
  \end{equation}
  is a normal extremal for problem $\Pbr$.

  \noindent By continuity, possibly restricting $\raggio > 0$ and $\cO$, we
  can assume, for any $r \in B(0, \rho)$, 
  \[
  \begin{split}
    & F_{101}^r\*\lambda^r(t)> 0 \quad \forall t \in [\^\tau_1 (r),
    \^\tau_2 (r) ], \\
    & u_1F_1^r\*\lambda^r(t)> 0 \quad \forall t \in [0, \^\tau_1 - \e] ,\\
    & \left( u_1 \Fr{001} + \Fr{101} \right)\*\lambda^r(t)> 0 \quad
    \forall t \in [\^\tau_1 - \e, \^\tau_1 + \e].
  \end{split}
  \]
  The Taylor expansion centered in $\tau_1(r)$ of the map $t \mapsto
  u_1F_1^r\*\lambda^r(t)$ proves that
  \begin{equation*}
    u_1F_1^r \*\lambda^r(t) = \dfrac{(t - \tau_1(r))^2}{2}\left(
      u_1 \Fr{001} + \Fr{101}
    \right)\*\lambda^r(\theta)
  \end{equation*}
  so that $u_1F_1^r \*\lambda^r(t) > 0$ for any $t \in [\^\tau_1 - \e,
  \tau_1(r))$.

  \noindent Analougous proof holds for the second bang arc.
\end{proof}

\subsection{Proof of Theorem \ref{thm:main}}

In order to prove Theorem \ref{thm:main}, we need to prove the strong
local optimality of the extremal pair defined in Lemma \ref{le:fond}.
We first prove that the extended second variation along the singular
arc of $\lambda^r$ is coercive, for sufficiently small $\norm{r}$.

Let $\tau_1(r)$ and $\tau_2(r)$ be the switching times of $\lambda^r$
as defined in Lemma \ref{le:fond} and let $v^r$ be the associated
singular control, i.e.~
\begin{equation*}
  v^r(t) := \dfrac{- \, \Fr{001}}{\Fr{101}}(\lambda^r(t)) \quad 
  t\in [\^\tau_1 - \e, \^\tau_2 + \e] .
\end{equation*}
Following the same lines as in the nominal problem $\Pbz$, let ${
  S^r}_{t}$ be the flow -- starting at the time $\tau_1(r)$ -- of the
vector field $f^r_{0} + v^r(t) f^r_1 $ and define $\gir{0}{t}$ and
$\gir{1}{t}$ as the dragged vector fields at time $\tau_1(r)$ along
such flow of the vector fields $f^r_{0}$ and $f^r_1$, respectively:
\begin{equation*} 
  \gir{i}{t}( x) := \left( { S^r}_{t*} \right)\inv f^r_i \* {
    S^r}_t( x)\, , \quad i=0, 1.
\end{equation*}
Let $x_{i}(r):= \xi^r(\tau_{i}(r))$, $i=1, 2$. Define coordinates
$y^r$ in a neighborhood of $x_1^r$ such that
\begin{equation*}
  y^r = y +O(r), \quad f_1^r \equiv \spd{}{y^r}, \ \text{ and } \ 
  f_0^r = \spd{}{y_2^r} - y_1^r \left( f^r_{01}(x_{i}(r)) + O(y^r)\right).
\end{equation*}
In such coordinates choosing $\beta^r(y^r) = - \dsum_{i=2}^n \mu_i^r
y_i^r $, where $\mu^r(\tau_1(r)) = \left( 0, \mu_2^r, \ldots, \mu_n^r
\right)$ the extended second variation along the singular arc of
$\lambda^r$ is the quadratic form
\begin{equation*} 
  {J^r}_\ext (\e_0, \e_1 , w) =
  \dfrac{1}{2}\int_{\tau_1(r)}^{\tau_2(r)} \left( w^2(t)
    F_{101}^r\*\lambda^r(t) + 2 w(t)\, \zeta(t) \cdot \dgir{1}{t}
    \cdot \beta^r((x_1(r)) \right) \ud t
\end{equation*}
on the linear sub--space $\cW^r$ of $\R^2 \times L^2([\tau_1(r),
\tau_2(r)], \R)$ of the triplets $(\e_0, \e_1, w)$ such that the
linear system
\begin{equation}\label{eq:sistjsefinalr}
  \begin{split}
    & \dot\zeta^r(t) = w(t) \dgir{1}{t}(x_1(r)), \\
    & \zeta^r(\tau_1(r)) =\e_0 f^r_0 (x_1(r)) + \e_1 \, f^r_1
    (x_1(r)), \quad \zeta^r(\tau_2(r)) = 0.
  \end{split}
\end{equation}
admits a solution $\zeta^r$.
\begin{lemma}
  \label{le:varsecr}
  Let $\lambda^r$ be the extremal pair of problem $\Pbr$ defined in
  Lemma \ref{le:fond}. There exists $\raggio > 0$ such that for any $r$,
  $\norm{r} < \raggio$, the extended second variation along the singular arc
  of $\lambda^r$ is coercive.
\end{lemma}
\begin{proof}
  Assume, by contradiction, there exists a sequence $r_{n} \to 0$ such
  that ${J^{r_{n}}_\ext}$ is not coercive on $\cW^{r_{n}}$.  Define
  $\tau_i^n := \tau_i(r_{n})$, $x_i^n := x_i^{r_n}(\tau_i(r_{n}))$,
  $i=1, 2$ and let $\e >0$ such that for any $n \in \N$, $[\tau_1^n ,
  \tau_2^n ] \subset I := [\^\tau_1 -\e, \^\tau_2 + \e]$. We extend
  any $w \in L^{2}([\tau_1^n , \tau_2^n ])$ to the interval $I$ by
  prolonging it as zero and we define $H := \R^2 \times L^2(I, \R)$.
  Then  there exists $\chi^{n} = \left( \e^{n}_{0}, \e^{n}_1,
    w^{n} \right)\in H$, $\norm{\chi^{n}} = 1$ such that
  \[
  \e_0^n f_0^{r_n}(x_1^n) + \e_1^n f_1^{r_n}(x_1^n) +
  \int_{\tau_1^n}^{\tau_2^n} w^n(t)\dgin{1}{t}(x_1^n)\ud t = 0, \qquad
  J^{r_{n}}_\ext (\chi^n) \leq 0.
  \]
  Without any loss of generality we can assume $\chi^n \rightharpoonup
  \chi^0 = \left( \e_0, \e_1, w_0 \right) \in H$, $\norm{\chi^0} \leq
  1$.  Let $\zeta^{n}$ be the associated solution of system
  \eqref{eq:sistjsefinalr}, for $r=r_{n}$.   By standard arguments
  \begin{equation*}
    \begin{split}
      \lim_{n\to\infty} \e_0^n f_0^{r_n}(x_1^n) + \e_1^n
      f_1^{r_n}(x_1^n)
      & + \int_{\tau_1^n}^{\tau_2^n} w^n(t)\dgin{1}{t}(x_1^n)\ud t = \\
      & =\e_0 f_0(\^x_1) + \e_1 f_1 (\^x_1) +
      \int_{\^\tau_1}^{\^\tau_2} w(t)\dgi{1}{t}(\^x_1)\ud t
    \end{split}
  \end{equation*}
  and
  \begin{equation}
    \label{eq:Jr1}
    \lim_{n\to\infty} \int_{\tau_1^n}^{\tau_2^n} \!\! \! \! 
    w^n(t) \, 
    \zeta^n(t) \cdot \dgin{1}{t}(x_1^n)
    \cdot \beta^r(x_1^n)
    \ud t = 
    \int_{\^\tau_1}^{\^\tau_2} \! \! w(t) \, 
    \zeta(t) \cdot \dgi{1}{t}(\^x_1)
    \cdot \beta(\^x_1) \ud t .
  \end{equation}
  Also
  \begin{equation}\label{eq:Jr2}
    \int_{I} w_{n}^{2}(t)R^{r_{n}}(t) \ud t = \int_{I}
    w_{n}^{2}(t)R(t) \ud t + \int_{I} w_{n}^{2}(t)\left( R^{r_{n}} -
      R \right)(t) \ud t .
  \end{equation}
  The second addendum converges to zero since $\norm{w_{n}}_2$ is
  uniformly bounded and $R^{r_{n}}$ converges to $R$ in the
  $L^{\infty}(I)$ norm. Let us turn to the first addendum:
  \begin{equation}\label{eq:Jr3}
    \begin{split}
      & \int_{I} w_{n}^{2}(t)R(t) \ud t = \int_{I} w_{0}^{2}(t)R(t)
      \ud t + \int_{I}
      \left( w_{n} - w_{0} \right)^{2}(t)R(t) \ud t + \\
      & + 2 \int_{I} R(t)w_{0}(t) \left( w_{n} - w_{0} \right) \ud t .
    \end{split}
  \end{equation}
  Letting $n \to \infty$ and summing up the results in
  \eqref{eq:Jr1}--\eqref{eq:Jr3} we obtain
  \begin{equation}
    \label{eq:Jr4}
    \liminf_{n \to \infty} J^{r_{n}}_{\ext} (\chi_{n})
    \geq C \norm{\chi_{0}}^{2} + \liminf_{n \to \infty}\int_{I}R(t)(w_{n} - w_{0})^{2}(t)
    \ud t
  \end{equation}
  If $\chi_{0} = 0$ then
  $\norm{w_{n}} \geq \frac{1}{2}$ for large enough $n$'s so that, by
  \eqref{eq:Jr4},
  \begin{equation*}
    \liminf_{n \to \infty} J^{r_{n}}_{\ext}(\chi_{n})
    \geq \frac{1}{2} \abs{I} \inf_{I}R(t) > 0.
  \end{equation*}
  By \cite{Hes66} this proves the coercivity of $J^{r_{n}}_{\ext}$. \\
If $\chi_0 \neq 0$, then equation \eqref{eq:Jr4} yields the claim, provided 
 $w_{0} (t)= 0$ \qo $t \in [\^\tau_1 -
  \e, \^\tau_1] \cup [\^\tau_2, \^\tau_2 + \e]$. Since $w_{n}
  \rightharpoonup w_{0}$ in $L^{2}([\^\tau_1 - \e, \^\tau_2 + \e])$,
  then $w_{n} \rightharpoonup w_{0}$ in $L^{2}([\^\tau_1 - \e,
  \^\tau_1])$.
  \begin{equation*}
    \int_{\^\tau_1 - \e}^{\^\tau_1}\abs{w_{n}(t)}\ud t =
    \begin{cases}
      0 \quad & \text{if } \^\tau_1 \leq \tau_1^{n}, \\
      \displaystyle\int_{\tau_1^{n}}^{\^\tau_1}\abs{w_{n}(t)}\ud t
      \quad & \text{if } \tau_1^{n} \leq \^\tau_1.
    \end{cases}
  \end{equation*}
  Since $\displaystyle\int_{ \tau_1^{n} }^{\^\tau_1}  \! \! \abs{w_{n}(t)}\ud
  t \leq \norm{w_{n}}_2\sqrt{\^\tau_1 - \tau_1^{n} } \leq
  \sqrt{\^\tau_1 - \tau_1^{n} } \to 0$ as $n\to \infty$, we get $w_{0}
  (t)= 0$ $\qo t \in [\^\tau_1 - \e, \^\tau_1]$.  Similarly one
  proves $w_{0} (t)= 0$ $\qo t \in [\^\tau_2, \^\tau_2 +\e]$.
\end{proof}

Lemma \ref{le:varsecr} proves (time, state)-local optimality of
$\xi^r$, see \cite{PS11}. To get state-local optimality of $\xi^r$ we
need to prove the following:
\begin{lemma}\label{le:simpler}
  Under Assumptions \ref{ass:PMP}--\ref{ass:coerc} there exists $\raggio > 0$
  such that for any $r$, $\norm{r} <  \raggio $ the trajectory $\xi^r$
  defined in Lemma \ref{le:fond} is injective.
\end{lemma}
\begin{proof}
  Assume, by contradiction, there exists a sequence $\{r_{n}\}$ that
  converges to zero and such that there exist $0 \leq t_1^{n} <
  t_2^{n} \leq T(r_{n})$ such that $\xi^{r_{n}}( t_1^{n}) =
  \xi^{r_{n}}( t_2^{n})$ i.e.
  \begin{equation}
    \label{eq:simple1}
    \int_{t_1^{n}}^{t_2^{n}} \left(
      f_{0}^{r_{n}}(\xi^{r_{n}}(s)) + u^{r_{n}}(s) f_1^{r_{n}}(\xi^{r_{n}}(s))
    \right) \ud s = 0. 
  \end{equation}
  Up to a subsequence we can assume $t_1^{n} \to \ol t_1$ and $t_2^{n}
  \to \ol t_2$ as $n \to \infty$, where $0 \leq \ol t_1 \leq \ol t_2
  \leq {\^ T}$.

  If $\ol t_1 < \ol t_2$, then passing to the limit in
  \eqref{eq:simple1} we get $\^{\xi}(\ol t_1) = \^{\xi}(\ol t_2)$, a
  contradiction. Hence we denote as $\ol t$ the common value of $\ol t_1$
  and $\ol t_2$.

  \noindent {\bfseries First case: \boldmath$0 \leq t_1^{n} <
    t_2^{n} \leq \tau_1(r_{n})$.}\\
  Applying the mean value theorem componentwise in \eqref{eq:simple1},
  for any $k=1, \ldots, n$ we get
  \begin{equation}
    \label{eq:simple2}
    \exists s_{k}^{n} \in [t_1^{n}, t_2^{n}] \colon
    \left( h_1^{r_{n}} \right)_{k}(\xi^{r_{n}}( s_{k}^{n} )) = 0.
  \end{equation}
  Letting $n \to \infty$ in \eqref{eq:simple2} we obtain $ h_1(\^\xi(
  \ol t )) = 0 $, a contradiction since $\ol t \leq \^\tau_1$ and
  $H_1(\^\lambda(t)) = 1 \quad \forall t \in [0, \^\tau_1]$.

  \noindent {\bfseries Second case: \boldmath$ t_1^{n} < \tau_1(r_{n})
    < t_2^{n} \leq
    \tau_2(r_{n})$.}\\
  In this case $\ol t = \^\tau_1$ and \eqref{eq:simple1} reads
  \begin{equation}
    \label{eq:simple4}
    \int_{t_1^{n}}^{\tau_1(r_{n})} 
    h_1^{r_{n}}(\xi^{r_{n}}(s)) \ud s =
    - \int_{\tau_1(r_{n})}^{t_2^{n}} \left(
      f_{0}^{r_{n}}(\xi^{r_{n}}(s)) + u^{r_{n}}(s) f_1^{r_{n}}(\xi^{r_{n}}(s))
    \right) \ud s .
  \end{equation}
  Since
  \begin{align*}
    & \lim_{n\to\infty}\dfrac{1}{\tau_1(r_{n}) - t_1^{n}}
    \int_{t_1^{n}}^{\tau_1(r_{n})}
    h_1^{r_{n}}(\xi^{r_{n}}(s)) \ud s = h_1(\xJ), \\
    & \lim_{n\to\infty}\dfrac{1}{ t_2^{n} - \tau_1(r_{n})}
    \int_{\tau_1(r_{n})}^{t_2^{n}} \!\!\! \left(
      f_{0}^{r_{n}}(\xi^{r_{n}}(s)) + u^{r_{n}}(s)
      f_1^{r_{n}}(\xi^{r_{n}}(s)) \right) \ud s = f_{0}(\xJ) +
    \^u(\^\tau_1+)f_1(\xJ)
  \end{align*}
  then, by \eqref{eq:simple4}, the ratio $\frac{ \left( t_2^{n} -
      \tau_1(r_{n}) \right)}{\tau_1(r_{n}) - t_1^{n}}$ converges to
  some quantity $L$ as $n\to \infty$ and
  \begin{equation*}
    h_1(\xJ) = -L \left(
      f_{0}(\xJ) + \^u(\^\tau_1+)f_1(\xJ)
    \right) 
  \end{equation*}
  i.e.
  \begin{equation*}
    (1 + L)f_{0}(\xJ) + \left( 1 - \^u(\^\tau_1+) \right)
    f_1(\xJ) = 0,
  \end{equation*}
  a contradiction since $f_{0}$ and $f_1$ are linearly independent at
  $\xJ$ and by the discontinuity of the reference control $\^u(t)$ at
  time $\^\tau_1$, see Remark \ref{re:refcont}.

  \noindent The other cases can be dealt with similarly. The case
  $t_1^{n}\leq \tau_1(r_{n}) < \tau_2(r_{n}) \leq t_2^{n}$ cannot
  occur since $t_2^{n} - t_1^{n} \to 0$ as $n \to \infty$ while
  $\tau_2(r_{n}) - \tau_1(r_{n}) \to \^\tau_2 - \^\tau_1 > 0$.
\end{proof}

\subsection{Proof of Theorem \ref{thm:uniq}}
\label{sec:pfuniq}
We now give the proof of the local uniqueness result stated in Theorem
\ref{thm:uniq}.  By Assumption \ref{ass:regjun}, there exists
$\ol\delta > 0$ such that both the maps
\[
\left. \left( u_1 F_{001} + F_{101} \right) \* \^\lambda
\right\vert_{[\^\tau_1 - \ol\delta , \^\tau_1]}\ \text{ and } \
\left. \left( u_2 F_{001} + F_{101} \right) \* \^\lambda
\right\vert_{[\^\tau_2 , \^\tau_2 + \ol\delta]}
\]
are strictly positive.  Without any loss of generality we can assume
$\ol\delta \in (0, \e)$, where $\e > 0$ is given in Lemma
\ref{le:fond}.  Thus the maps
\[
\left. u_1 F_{01}\*\^\lambda\right\vert_{[\^\tau_1 - \ol\delta ,
  \^\tau_1]} \ \text{ and } \ \left. u_2
  F_{01}\*\^\lambda\right\vert_{[\^\tau_2, , \^\tau_2 + \ol\delta]}
\]
are strictly monotone increasing. For any $\delta \in [0, \ol\delta]$
set
\begin{equation}
  \label{eq:M1m1a1}
  \begin{split}
    \Maxdd(\delta) & = \max\left\{ \left( u_1 F_{001} + F_{101}
      \right) \* \^\lambda(t) \colon t \in [\^\tau_1 - \delta ,
      \^\tau_1]
    \right\} ,\\
    \mindd(\delta) & = \min\left\{ \left( u_1 F_{001} + F_{101}
      \right) \* \^\lambda(t) \colon t \in [\^\tau_1 - \delta ,
      \^\tau_1]
    \right\} , \\
    \minbang(\delta)& = \min\left\{ \left( u_1 F_1 \right) \*
      \^\lambda(t) \colon t \in [0, \^\tau_1 - \delta ] \right\} ,
  \end{split}
\end{equation}
Then, a Taylor expansion of $u_1 F_1\* \^\lambda(t) $ in $t =\^\tau_1$
yields, for any $t \in [\^\tau_1 - \delta , \^\tau_1]$ the
inequalities
\begin{align}
  \label{eq:M1m1F01}
  -\delta \Maxdd(\delta) \leq \Maxdd(\delta) \left( t - \^\tau_1
  \right) \leq & u_1 F_{01} \* \^\lambda(t) \leq \mindd(\delta) \left(
    t -
    \^\tau_1 \right) , \\
  \label{eq:M1m1F1}
  \dfrac{\mindd(\delta)\left( t - \^\tau_1 \right)^2}{2} \leq & u_1
  F_1 \* \^\lambda(t) \leq \dfrac{\Maxdd(\delta)\left( t - \^\tau_1
    \right)^2}{2} \leq \dfrac{\Maxdd(\delta)\, \delta^2}{2} .
\end{align}
Moreover without any loss of generality we can assume
$ \argmin \left. u_1 F_1\*\^\lambda(t) \right\vert_{ [0, \^\tau_1 -
  \delta ]} = \^\tau_1 - \delta, $
so that
\begin{equation}
  \label{eq:m1aM1}
  \dfrac{\mindd(\delta)\delta^2}{2} \leq \minbang(\delta)\leq
  \dfrac{\Maxdd(\delta)\, \delta^2}{2} .
\end{equation}
Define
\begin{equation}
  \label{eq:ThetauM}
  \begin{split}
    \Theta & := \min\left\{ F_{101}\*\^\lambda(t) \colon t \in
      [\^\tau_1, \^\tau_2]
    \right\} , \\
    u_M & := \sup\left\{ \abs{\^u(t)} \colon t \in (\^\tau_1,
      \^\tau_2) \right\} = \sup\left\{
      \abs{\dfrac{F_{001}}{F_{101}}\*\^\lambda(t)} \colon t \in
      (\^\tau_1, \^\tau_2) \right\} < 1.
  \end{split}
\end{equation}
Similarly, set
\begin{equation}
  \label{eq:m2M2a2}
  \begin{split}
    \Maxddse(\delta) & = \max\left\{ \left( u_2 F_{001} + F_{101}
      \right) \* \^\lambda(t) \colon t \in [\^\tau_2 , \^\tau_2 +
      \delta]
    \right\} ,\\
    \minddse(\delta) & = \min\left\{ \left( u_2 F_{001} + F_{101}
      \right) \* \^\lambda(t) \colon t \in [\^\tau_2 , \^\tau_2 +
      \delta]
    \right\} , \\
    \minbangse(\delta) & = \min \left\{ u_2 F_1 \*\^\lambda(t) \colon
      t \in [\^\tau_2 + \delta, \^T] \right\}.
  \end{split}
\end{equation}
Again, a Taylor expansion of $u_2 F_1\* \^\lambda(t) $ in $t
=\^\tau_2$ yields, for any $t \in [\^\tau_2 , \^\tau_2 + \delta]$
\begin{align}
  \label{eq:M1m1F01bis}
  \minddse(\delta) \left( t - \^\tau_2 \right) \leq & u_2 F_{01} \*
  \^\lambda(t) \leq \Maxddse(\delta) \left( t -
    \^\tau_2 \right) \leq \Maxddse(\delta) \delta , \\
  \label{eq:M1m1F1bis}
  \dfrac{\minddse(\delta)\left( t - \^\tau_2 \right)^2}{2} \leq & u_2
  F_1 \* \^\lambda(t) \leq \dfrac{\Maxddse(\delta)\left( t - \^\tau_2
    \right)^2}{2} \leq \dfrac{\Maxddse(\delta)\, \delta^2}{2} .
\end{align}
For any $\delta \in (0, \ol\delta)$ choose $\cO_\delta(\lri) \subset
\fibcot$ such that for any $\ell \in \cO_\delta(\lri) $ the following
inequalities hold:
\begin{equation}
  \label{eq:disugua1}
  \begin{alignedat}{2}
    & u_1 F_1\*\^\cF_t(\ell) \geq \dfrac{\minbang(\delta)}{2} \qquad
    && t
    \in [0, \^\tau_1 - \delta] \\
    & \abs{u_1 F_{01}\*\^\cF_t(\ell)} < 2 \delta \Maxdd(\delta )
    \qquad
    && t \in [\^\tau_1 - \delta, \^\tau_2 + \delta] \\
    & \abs{u_1 F_1\*\^\cF_t(\ell)} < \delta^2 \Maxdd(\delta ) \qquad
    && t \in [\^\tau_1 - \delta, \^\tau_2 + \delta] \\
    & \abs{\dfrac{F_{001}}{F_{101}}\*\^\cF_t(\ell)} < \dfrac{1 + 3
      u_M}{4} \qquad
    && t \in [\^\tau_1 - \delta, \^\tau_2 + \delta] \\
    & \abs{F_{101}\*\^\cF_t(\ell)} \geq \dfrac{\Theta}{2} \qquad
    && t \in [\^\tau_1 - \delta, \^\tau_2 + \delta] \\
    & u_2 F_1\*\^\cF_t(\ell) \geq \dfrac{\minbangse(\delta)}{2} \qquad
    && t \in [\^\tau_2 + \delta, \^T + \delta] .
  \end{alignedat}
\end{equation}
Set
\begin{equation}
  \label{eq:vdelta}
  \cV_\delta := \left\{ 
    \left( t, \^\cF_t(\ell) \right) \colon (t, \ell) \in \left[0, \^T +
      \delta \right] \times \cO_\delta(\lri)
  \right\}
\end{equation}
We choose $\raggio_\delta >0 $ such that for any $r \colon \abs{r } \leq
\raggio_\delta$ the followings hold in $\cV_\delta$
\begin{equation}
  \label{eq:disugua2}
  \begin{alignedat}{2}
    & u_1 \Fr{1}(\ell) \geq \dfrac{\a_1(\delta)}{4} \qquad &&
    \text{if } t \leq \^\tau_1 - \delta \\
    & \abs{\Fr{01}(\ell)} \leq 4 \delta \Maxdd(\delta ) \qquad
    && \text{if } t \in [\^\tau_1 - \delta, \^\tau_2 + \delta] \\
    & \abs{\Fr{1}(\ell)} \leq 2\delta^2 \Maxdd(\delta ) \qquad
    && \text{if } t \in [\^\tau_1 - \delta, \^\tau_2 + \delta] \\
    & \abs{\dfrac{\Fr{001}}{\Fr{101}}(\ell)} < \dfrac{1 + u_M}{2}
    \qquad
    && \text{if } t \in [\^\tau_1 - \delta, \^\tau_2 + \delta] \\
    & \abs{\Fr{101}(\ell)} \geq \dfrac{\Theta}{4} \qquad
    && \text{if } t \in [\^\tau_1 - \delta, \^\tau_2 + \delta] \\
    & u_2 \Fr{1}(\ell) \geq \dfrac{\minbangse(\delta)}{4} \qquad &&
    \text{if } t \in [\^\tau_2 + \delta, \^T + \delta] .
  \end{alignedat}
\end{equation}
An easy consequence of \eqref{eq:disugua2} is
\begin{equation}
  \label{eq:disugua2bis}
  \begin{split}
    \left( \Fr{101} \pm \Fr{001}\right) (\ell) \geq F_{101}^r (\ell)
    \left( 1- \abs{\dfrac{ F_{001}^r}{ F_{101}^r}(\ell) } \right) \geq
    \dfrac{\Theta(1 - u_M)}{8} \qquad\qquad\qquad \\
    \text{if } ( t, \ell) \in \cV_\delta , \quad t \in \left[ \^\tau_1
      - \delta, \^\tau_2 + \delta \right], \quad \abs{r} < \raggio_\delta.
  \end{split}
\end{equation}
Let $\~\lambda \colon \left[ 0, \~T \right] \to \fibcot$ be an
extremal of $\Pbr$ such that $\abs{\~T - \^T } < \e$ and whose graph
is in $\cV_\delta$. Let $\~u \colon \left[ 0, \~T \right] \to [-1, 1]$
be the associated control. We want to prove that $\~T =T^r$,
$\~\lambda \equiv \lambda^r$ and $\~u \equiv u^r$.

The proof is split in several steps. First we prove that the
trajectory of $\~\lambda$ intersects $\Sigmar$. Then we show that the
entry time in $\Sigmar$ is in $(\^\tau_1 - \e, \^\tau_1 + \e)$ and
that the trajectory remains on $\Sigmar$ at least untile time
$\^\tau_2 - \e$. Finally we prove that once $\~\lambda$ has left
$\Sigmar$, it remains bang till the final time $\~T$.

\noindent{\bfseries Step 1: {\boldmath $\Fr{1}\*\~\lambda(t)$
    annihilates for
    some $t\in \left[0, \~T\right]$.}} \\
Assume by contradiction that $\Fr{1}\*\~\lambda(t) $ never
annihilates. Since $\~\lambda(0)$ is close to $\lri$, we must have
$\~u(t) \equiv u_1$ for any $t \in \left[0, \~T \right]$. Thus
\begin{equation*}
  \begin{split}
    & u_1 \Fr{1}\*\~\lambda(\^\tau_2) = u_1
    \Fr{1}\*\~\lambda(\^\tau_1)
    + \int_{\^\tau_1}^{\^\tau_2} u_1 \Fr{01}\*\~\lambda(s) \ud s= \\
    & = u_1 \Fr{1}\*\~\lambda(\^\tau_1) + \int_{\^\tau_1}^{\^\tau_2}
    \left( u_1 \Fr{01}\*\~\lambda(\^\tau_1) + \int_{\^\tau_1}^s \left(
        \Fr{101} + u_1 \Fr{001} \right)\*\~\lambda(a) \ud a
    \right) \ud s \geq \\
    & \geq u_1 \Fr{1}\*\~\lambda(\^\tau_1) + \left( \^\tau_2 -
      \^\tau_1 \right) u_1 \Fr{01}\*\~\lambda(\^\tau_1) +
    \dfrac{\Theta (1 -u_M) (\^\tau_2 - \^\tau_1)^2}{16} \geq \\
    & \geq - \left(\^\tau_2 - \^\tau_1 \right) 4 \, \delta \,
    \Maxdd(\delta) + \dfrac{\Theta (1 -u_M) (\^\tau_2 -
      \^\tau_1)^2}{16} > 2\, \delta^{2}\Maxdd(\delta),
  \end{split}
\end{equation*}
if $\delta$ is choosen small enough.  A contradiction of
\eqref{eq:disugua2}. Define
\begin{equation*}
  \~\tau_1 := \inf \left\{ t \in \left[ 0, \~T \right] \colon \Fr{1}\*\~\lambda(t) = 0 \right\} 
\end{equation*}
so that
\begin{equation} \label{eq:lambda1tilda} \~\lambda(t) = \exp t
  \vHr{1}(\~\lambda(0)) \quad \forall t \in [0, \~\tau_1], \qquad
  \Fr{1}\*\~\lambda(\~\tau_1) = 0 , \quad u_1
  \Fr{01}\*\~\lambda(\~\tau_1) \leq 0 .
\end{equation}

\noindent{\bfseries Step 2: {\boldmath $\~\tau_1 \in \left( \^\tau_1 -
      \e,
      \^\tau_1 + \e\right)$, $\e$ defined in Lemma \ref{le:fond}.}}\\
By definition of $\cV_\delta$, $u_1\Fr{1}\*\~\lambda(t) \geq
\dfrac{\a_1(\delta)}{4}$ if $t \leq \^\tau_1 - \delta$, so that
$\~\tau_1 \geq \^\tau_1 - \delta > \^\tau_1 - \e$ 
since $0 < \delta < \ol\delta < \e$.  If $\~\tau_1 \leq \^\tau_1$ we
are done. Otherwise, let $s := \~\tau_1 - \^\tau_1 > 0$. A Taylor
expansion in $\~\tau_1$ gives
\begin{equation*}
  \begin{split}
    & u_1 \Fr{1}\*\~\lambda(\^\tau_1 ) = u_1 \Fr{1}\*\~\lambda(\~\tau_1 - s) = \\
    & = -s \, u_1 \Fr{01}\*\~\lambda(\~\tau_1) + \dfrac{s^2}{2}\left(
      u_1 \Fr{001} + \Fr{101} \right)\*\~\lambda(a), \qquad \text{for
      some } a \in (\^\tau_1, \~\tau_1 )
  \end{split}
\end{equation*}
Hence, by \eqref{eq:disugua2} and \eqref{eq:lambda1tilda},
\begin{equation*}
  \begin{split}
    s & = \dfrac{u_1 \Fr{01}\*\~\lambda(\~\tau_1) + \sqrt{ \left( u_1
          \Fr{01}\*\~\lambda(\~\tau_1) \right)^2 \!\! + \left(\left(
            u_1 \Fr{001} + \Fr{101} \right)\*\~\lambda(a) \right) 2 \,
        u_1 \Fr{1}\*\~\lambda(\^\tau_1 ) }}
    {\left( u_1 \Fr{001} + \Fr{101} \right)\*\~\lambda(a)} \\
    & \leq \sqrt{ \left( \dfrac{ u_1 \Fr{01}\*\~\lambda(\~\tau_1) }{
          \left( u_1 \Fr{001} + \Fr{101} \right)\*\~\lambda(a) }
      \right)^2 + \dfrac{ 2 \, u_1 \Fr{1}\*\~\lambda(\^\tau_1) }{
        \left( u_1 \Fr{001} + \Fr{101} \right)\*\~\lambda(a) }
    } \leq \\
    & \leq \sqrt{ \left(\dfrac{ 8\, \delta \Maxdd(\delta) }{\Theta(1-
          u_M) }\right)^2 + \dfrac{ 16\, \delta^2 \Maxdd(\delta)
      }{\Theta(1- u_M) } } = \dfrac{4 \, \delta\,
      \sqrt{\Maxdd(\delta)}}{\Theta(1 - u_M)} \sqrt{ 4\,
      \Maxdd(\delta) + \Theta (1- u_M) }
  \end{split}
\end{equation*}
Therefore $s < \e$, if $\delta$ is choosen small enough.
 
\noindent{\bfseries Step 3: {\boldmath$ \Fr{1}\*\~\lambda(t) \equiv 0$
    for any $t \in [\~\tau_1, \^\tau_2 - \e]$} }.\\ Let
\begin{equation*}
  \cA := \left\{
    t \in (\~\tau_1 , \^\tau_2 + \delta) \colon \Fr{1}\*\~\lambda(t) \neq 0
  \right\}.
\end{equation*}
$\cA$ is open, hence it contains at least an open interval. Let $I =
(t_1, t_2)\subset \cA$ be a maximal interval. Then $
\Fr{1}\*\~\lambda(t_1) = 0$ and the control $\~u(t)$ is constant in
$I$: $\left. \~u (t) \right\vert_{I} = \~u_I := \left. \sgn{
    \Fr{1}\*\~\lambda(t) }\right\vert_I$, so that
\begin{equation*}
  \~\lambda(t) = \exp (t - t_1) \left(\vFr{0} + \~u_1\vFr{1} \right) \*\~\lambda(t_1) 
  \quad \forall t \in [t_1, t_2], \quad \text{ and } 
  \~u_I \Fr{01}\*\~\lambda(t_1) \geq 0.
\end{equation*}
For any $t \in [t_1, t_2]$ we get
\begin{equation}\label{eq:intermaxi}
  \begin{split}
    & \~u_I \Fr{1}\*\~\lambda(t) = \int_{t_1}^{t} \~u_I
    \Fr{01}\*\~\lambda(s) \ud s
    \geq \~u_I \Fr{01}\*\~\lambda(t_1)  + \\
    & + \int_{t_1}^{t} \ud s \int_{t_1}^s \left( \~u_I \Fr{001} +
      \Fr{101} \right)\*\~\lambda(a) \ud a \geq
    \dfrac{\Theta(1-u_M)\left(t - t_1\right)^2}{16} .
  \end{split}
\end{equation}
Two cases may occur:

{\em First case } $I = (t_1, t_2)$ for some $t_1 < t_2 < \^\tau_2 +
\delta $. \\
In this case $\Fr{1}\*\~\lambda(t_2) = 0$. Choosing $t = t_2$ in
\eqref{eq:intermaxi} we get a contradiction.  This shows that if
$\~\lambda$ leaves $\Sigmar$ before time $\^\tau_2 + \delta$, then it
remains out of $\Sigmar$, at least until time $\^\tau_2 + \delta$.

{\em Second case } $I = (t_1, \^\tau_2 + \delta)$ for some $t_1 <
\^\tau_2 + \delta$. We need to show that $t_1 > \^\tau_2 -
\e$. Assume, by contradiction, that $t_1 \leq \^\tau_2 - \e$. Choosing
$t = \^\tau_2$ in \eqref{eq:intermaxi} and by choosing a small enough
$\delta$ we get
\begin{equation*}
  \~u_1 \Fr{1}\*\~\lambda(\^\tau_2) \geq \dfrac{\Theta(1 - u_M)
    \e^2}{16}  > 2\delta^2 \Maxdd(\delta),
\end{equation*}
a contradiction.  Let
\[
\~\tau_2 := \max \left\{t \in [ \~\tau_1 , \~T ] \colon
  \Fr{1}\*\~\lambda(s) = 0 \quad \forall s \in [\~\tau_1, t]\right\}.
\]
The two cases above prove that $\~\lambda(t) \in \Sigmar$ for any $t
\in [\~\tau_1, \^\tau_2 - \e]$ so that $\~\tau_2 \geq \^\tau_2 - \e$.
If $\~\tau_2 \geq \^\tau_2 + \delta$, then $
\Fr{1}\*\~\lambda(\^\tau_2 + \delta) = 0$, a contradiction by
\eqref{eq:disugua2}. Thus, $\~\tau_2 < \^\tau_2 + \delta < \^\tau_2 +
\e$.

\noindent{\bfseries Step 4: {\boldmath $\~\tau_2 \leq \^\tau_2 + \e$ and, for any $t \in ( \~\tau_2 , \~T)$, $\~\lambda(t) \notin \Sigmar$ and $\~u(t) \equiv u_2$}. }\\
By \eqref{eq:disugua2} and the previous step, $\Fr{1}\* \~\lambda(t)$
is non zero for any $t \in (t_2, \~T]$. Hence its sign is constant and
$\left. \~u(t) \right\vert_{(t_2, \~T]} = \~u_2 :=
\left. \sgn{\Fr{1}\* \~\lambda(t)} \right\vert_{(t_2, \~T]}$. By
\eqref{eq:disugua2} $u_2 \Fr{1}\*\~\lambda(t)$ is positive, hence
$\~u_2 = u_2$.

Since $\~\lambda$ is a bang--singular--bang extremal satisfying the
claims of Lemma \ref{le:fond}, then $\~\lambda = \lambda^r$.



\begin{thebibliography}{10}

\bibitem{AS04}
Andrei~A. Agrachev and Yuri~L. Sachkov.
\newblock {\em Control Theory from the Geometric Viewpoint}.
\newblock Springer-Verlag, 2004.

\bibitem{Arn80}
Vladimir~I. Arnold.
\newblock {\em Mathematical Methods in Classical Mechanics}.
\newblock Springer, New York, 1980.

\bibitem{Conti76}
Roberto Conti.
\newblock {\em Linear differential equations and control}, volume~I of {\em
  Institutiones Mathematicae}.
\newblock Istituto Nazionale di Alta Matematica, Roma, 1976.
\newblock Distributed by Academic Press Inc.

\bibitem{Cra95}
B.~D. Craven.
\newblock {\em Control and optimization}.
\newblock Chapman \& Hall, 1995.

\bibitem{Fel11}
Ursula Felgenhauer.
\newblock Controllability and stability for problems with bang-singular-bang
  optimal control.
\newblock Private Communication.

\bibitem{Fel04}
Ursula Felgenhauer.
\newblock Optimality and sensitivity for semilinear bang-bang type optimal
  control problems.
\newblock {\em Int. J. Appl. Math. Comput. Sci.}, 14(4):447--454, 2004.

\bibitem{Fel12}
Ursula Felgenhauer.
\newblock Structural stability investigation of bang-singular-bang optimal
  controls.
\newblock {\em Journal of Optimization Theory and Applications}, 152:605--631,
  2012.
\newblock 10.1007/s10957-011-9925-0.

\bibitem{FPS09}
Ursula Felgenhauer, Laura Poggiolini, and Gianna Stefani.
\newblock Optimality and stability result for bang--bang optimal controls with
  simple and double switch behaviour.
\newblock {\em CONTROL AND CYBERNETICS}, 38(4B):1305 -- 1325, 2009.

\bibitem{GK72}
V.~{Gabasov} and F.M. {Kirillova}.
\newblock High order necessary conditions for optimality.
\newblock {\em SIAM J. Control Optimization}, 10:127--188, 1972.

\bibitem{Hes66}
Magnus~R. Hestenes.
\newblock {\em Calculus of Variations and Optimal Control Theory}.
\newblock John Wiley \& Sons, New York, New York, 1966.

\bibitem{Mln93}
K.~Malanowski.
\newblock Two-norm approach in stability and sensitivity analisys of
  optimization and optimal control problems.
\newblock {\em Advances in Math. Sciences and Applications}, 2:397--443, 1993.

\bibitem{Mln94}
K.~Malanowski.
\newblock Regularity of solutions in stability analisys of optimization and
  optimal control problems.
\newblock {\em Control and Cybernetics}, 23:61--86, 1994.

\bibitem{Mln01}
K.~Malanowski.
\newblock Stability and sensitivity analysis for optimal control problems with
  control-state constraints.
\newblock In {\em Dissertationes Mathematicae}, volume CCCXCIV. Institute of
  Mathematics, Polish Academy of Sciences, 2001.

\bibitem{PSp08}
Laura Poggiolini and Marco Spadini.
\newblock Sufficient optimality conditions for a bang-bang trajectory in a
  bolza problem.
\newblock In Andrey Sarychev, Albert Shiryaev, Manuel Guerra, and Maria
  do~Ros\'ario Grossinho, editors, {\em Mathematical Control Theory and
  Finance}, pages 337--357. Springer Berlin Heidelberg, 2008.
\newblock 10.1007/978-3-540-69532-5\_19.

\bibitem{PSp11}
Laura Poggiolini and Marco Spadini.
\newblock Strong local optimality for a bang-bang trajectory in a mayer
  problem.
\newblock {\em SIAM Journal on Control and Optimization}, 49(1):140--161, 2011.

\bibitem{PS08}
Laura Poggiolini and Gianna Stefani.
\newblock Sufficient optimality conditions for a bang--singular extremal in the
  minimum time problem.
\newblock {\em Control and Cybernetics}, 37(2):469 -- 490, 2008.

\bibitem{PS11}
Laura Poggiolini and Gianna Stefani.
\newblock Bang-singular-bang extremals: sufficient optimality conditions.
\newblock {\em Journal of Dynamical and Control Systems}, 17:469--514, 2011.
\newblock 10.1007/s10883-011-9127-y.

\bibitem{PS12}
Laura Poggiolini and Gianna Stefani.
\newblock On the minimum time problem for dodgem car-like bang-singular
  extremals.
\newblock In Ivan Lirkov, Svetozar Margenov, and Jerzy Wasniewski, editors,
  {\em Large-Scale Scientific Computing}, volume 7116 of {\em Lecture Notes in
  Computer Science}, pages 147--154. Springer Berlin / Heidelberg, 2012.
\newblock 10.1007/978-3-642-29843-1\_16.

\bibitem{PS13}
Laura Poggiolini and Gianna Stefani.
\newblock A case study in strong optimality and structural stability of
  bang-singular extremals.
\newblock To appear.

\bibitem{Ste07}
Gianna Stefani.
\newblock Strong optimality of singular trajectories.
\newblock In Fabio Ancona, Alberto Bressan, Piermarco Cannarsa, Francis Clarke,
  and Peter~R. Wolenski, editors, {\em Geometric Control and Nonsmooth
  Analysis}, volume~76 of {\em Series on Advances in Mathematics for Applied
  Sciences}, pages 300--326, Hackensack, NJ, 2008.
  World~Scientific~Publishing~Co.~Pte.~Ltd.
\newblock pp. 361 ISBN: 978-981-277-606-8.

\bibitem{SZ97}
Gianna Stefani and PierLuigi Zezza.
\newblock Constrained regular {L}{Q}-control problems.
\newblock {\em SIAM J. Control Optim.}, 35(3):876--900, 1997.

\bibitem{ZB94}
Michail~I. Zelikin and Vladimir~F. Borisov.
\newblock {\em Theory of Chattering Control}.
\newblock Systems \& Control: Foundations \& Applications. Birkhauser, Boston,
  Basel, Berlin, 1994.

\end{thebibliography}

\end{document}